\documentclass[reqno]{amsart}
\usepackage{amscd,amssymb,amsmath,amsthm}
\usepackage{psfrag}
\usepackage{color,xcolor}
\usepackage{hyperref}
\usepackage{caption,subcaption}
\usepackage[percent]{overpic}
\usepackage[textwidth=6in, textheight=8.5in, centering]{geometry}
\def\cM{\mathcal{M}}
\def\cF{\mathcal{F}}
\def\cS{\mathcal{S}}

\def\cB{\mathcal{B}}

\def\cW{\mathcal{W}}

\def\cR{\mathcal{R}}

\def\bZ{\mathbb{Z}}

\def\eps{\varepsilon}

\def\ds{\displaystyle}

\newtheorem{thm}{Theorem}
\newtheorem{lem}{Lemma}

\newtheorem{pro}{Proposition}

\theoremstyle{definition}
\newtheorem{remark}{Remark}
\newtheorem{definition}{Definition}
\newtheorem{conj}{Conjecture}
\def\beq{\begin{equation}}
\def\eeq{\end{equation}}
\numberwithin{equation}{section}

\begin{document}

\title[Hyperbolicity of  asymmetric lemon billiards]
{On another edge of defocusing: hyperbolicity of asymmetric lemon billiards}

\author{Leonid Bunimovich}
\address{School of Mathematics, Georgia Institute of Technology, Atlanta, GA 30332}
\email{bunimovh@math.gatech.edu}

\author{Hong-Kun~Zhang}
\address{Department of Mathematics and Statistics, UMass Amherst, Amherst, MA 01003}
\email{hongkun@math.umass.edu}

\author{Pengfei~Zhang}
\address{Department of Mathematics, University of Houston, Houston, TX 77004}
\email{pzhang@math.uh.edu}



\begin{abstract}
Defocusing mechanism provides a way to construct chaotic (hyperbolic)
billiards with focusing components by separating all regular components of
the boundary of a billiard table sufficiently {\it far away} from each
focusing component. If all focusing components of the boundary of the
billiard table are circular arcs, then the above separation requirement
reduces to that all circles obtained by completion of focusing components are
contained in the billiard table. In the present paper we demonstrate that a
class of convex tables--{\it asymmetric lemons}, whose boundary consists of
two circular arcs, generate hyperbolic billiards. This result is quite
surprising because the focusing components of the asymmetric lemon table are
{\it extremely close} to each other, and because these tables are
perturbations of the first convex ergodic billiard constructed more than
forty years ago.
\end{abstract}

\maketitle


\section{Introduction}

Billiards are dynamical systems generated by the motion of a point particle
along the geodesics on a compact Riemannian manifold $Q$ with boundary. Upon
hitting the boundary of $Q$, the particle changes its velocity according to
the law of elastic reflections. The studies of chaotic billiard systems were
pioneered by Sina\v{\i} in his seminal paper \cite{Si70} on dispersing
billiards. A major feature of billiards which makes them arguably the most
visual dynamical systems is that all their dynamical and statistical
properties are completely determined by the shape of the billiard table $Q$
and in fact by the structure of the boundary $\partial Q$.

Studies of convex billiards which started much earlier demonstrated that the
convex billiards have  regular dynamics and are even integrable. Such
examples are billiards in circles or in squares, which everybody  studied
(without knowing that they study billiards) in a middle or in a high school.
Jacobi proved integrability of billiards in ellipses by introducing
elliptical coordinates in which the equations of motion are separated.
Birkhoff conjectured that ellipses are the only integrable two dimensional
smooth convex tables which generate completely integrable billiards. Later
Lazutkin \cite{La} proved that all two-dimensional convex billiards with
sufficiently smooth boundary admit caustics and hence they can not be ergodic
(see also \cite{Dou82}).

The first examples of hyperbolic and ergodic billiards with dispersing as
well as with focusing components were constructed in \cite{Bu74A}. A closer
analysis of these examples allowed one to realize that there is another
mechanism of chaos (hyperbolicity)  than the mechanism of dispersing which
generates hyperbolicity in dispersing billiards.  This makes it possible to
construct hyperbolic and ergodic billiards which do not have dispersing
components on the boundary \cite{Bu74B,Bu79}. Some billiards on convex tables
also belong to this catergory. The {\it first} one was a table with boundary
component consisting of a major arc and a chord connecting its two end
points. Observe that this billiard is essentially equivalent to the one
enclosed by two circular arcs symmetric with respect to the cutting chord.
This billiard belongs to the class of (chaotic) flower-like billiards. The
boundaries of these tables have the smoothness of order $C^0$. The stadium
billiard (which became strangely much more popular than the others) appeared
as one of many examples of convex ergodic billiards with smoother ($C^1$)
boundary. Observe that in a flower-like billiard, all the circles generated
by the corresponding petals (circular arcs) completely lie within the
billiard table (flower).

In the present paper we consider (not necessarily small) perturbations of the
first class of chaotic focusing billiards, i.e., with boundary made of a
major arc and a chord.  While keeping the major arc, we replace the chord (a
neutral or zero curvature component of the boundary) by a circular arc with
smaller curvature. This type of billiards were constructed in \cite{{CMZZ}},
with certain numerical results. We prove rigourously that the corresponding
billiard tables generate hyperbolic billiards  under the conditions that the
chord is not too long, and the new circular arc has sufficiently small
curvature. More precisely, we assume that the length of the chord does not
exceed the radius of the circular component of the boundary, see Theorem
\ref{main}. This condition  is a purely technical one, and we conjecture that
the hyperbolicity holds without this restriction.

It is worthwhile to recall that defocusing mechanism of chaos was one of a
few examples where discoveries of new mechanisms/laws of nature were made in
mathematics rather than in physics. No wonder that physicists did not believe
that, even though rigorous proofs were present, until they check it
numerically. After that, stadia and other focusing billiards were built in
physics labs over the world.  However, intuitive ``physical" understanding of
this mechanism always was that defocusing must occur after any reflection off
the focusing part of the boundary. Indeed, there are just a few very special
classes of chaotic (hyperbolic) billiards where this condition was violated
(see e.g. \cite{BD}). However, these billiards were specially constructed to
get hyperbolic billiards. To the contrary, hyperbolicity of asymmetric lemons
came as a complete surprise to everybody. This class of chaotic billiards
forces physicists as well as mathematicians to reconsider their understanding
of this fundamental mechanism of chaos.

Our billiards can also be viewed as far-reaching generalizations of classical
lemon-type billiards. The lemon billiards were introduced by Heller and
Tomsovic \cite{HeTo} in 1993, by taking the intersection of two unit disks,
while varying the distance between their centers, say $b$. This family of
billiards have been extensively studied numerically in physics literature in
relation with the problems of quantum chaos (see \cite{MHA,ReRe}). The
coexistence of the elliptic islands and chaotic region has also been observed
numerically for all of the lemon tables as long as $b\neq 1$. Therefore, the
lemon table with $b=1$ is the {\it only} possible billiard system with
complete chaos in this family. See also \cite{LMR99,LMR01} for the studies of
classical and quantum chaos of lemon-type billiards with general quadric
curves.

\begin{figure}[h]
\begin{overpic}[width=2.5in]{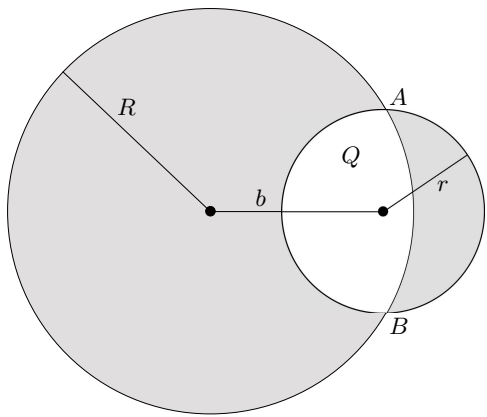}
\put(70,53){\small$Q$}
\put(90,47){\small$r$}
\put(23,63){\small$R$}
\put(52,44){\small$b$}
\put(80,65){\small$A$}
\put(80,17){\small$B$}
\end{overpic}
\caption{Basic construction of an asymmetric lemon table $Q(b,R)$. }
\label{basic}
\end{figure}

The lemon tables were embedded into a 3-parameter family--the {\it asymmetric
lemon billiards} in \cite{CMZZ}, among which the ergodicity is {\it no
longer} an exceptional phenomenon. More precisely, let $Q(r,b,R)$ be the
billiard table obtained as the intersection of a disk $D_r$ of radius $r$
with another disk $D_R$ of radius $R> r$, where $b>0$ measures the distance
between the centers of these two disks (see Fig.~\ref{basic}). Without loss
of generality, we will assume $r=1$ and denote the lemon table by
$Q(b,R)=Q(1,b,R)$. Restrictions on $b$ and $R$ will be specified later on to
ensure the hyperbolicity of the billiard systems on these asymmetric lemon
tables. On one hand, these billiard tables have extremely simple shape, as
the boundary of the billiard table $Q(b,R)$ only consists of two circular
arcs. Yet on the other hand, these systems already exhibit rich dynamical
behaviors, as it has been numerically observed in \cite{CMZZ} that there
exists an infinite strip $\mathcal{D}\subset [1,\infty)\times [0,\infty)$,
such that for any $(b,R)\in \mathcal{D}$, the billiard system on $Q(b,R)$ is
ergodic.

In this paper we give a rigorous proof of the hyperbolicity on a class of
asymmetric lemon billiards $Q(b,R)$. Our approach is based on the analysis of
continued fractions generated by the billiard orbits, which were introduced
by Sina\v{\i} \cite{Si70}, see also \cite{Bu74A}. Continued fractions are
intrinsic objects for billiard systems, and therefore they often provide
sharper results than those one gets by the abstract cone method, which deals
with hyperbolic systems of any nature and does not explore directly some
special features of billiards. In fact, already in the fundamental paper
\cite{Si70} invariant cones were immediately derived from the structure of
continued fractions generated by dispersing billiards. The study of
asymmetric lemon billiards demonstrates that defocusing mechanism can
generate chaos in much more general setting than it was thought before.

\subsection{Main results}

Let $b>0$ and $R>1$ be two positive numbers, $Q=Q(b,R)$ be the asymmetric
lemon table obtained by intersecting the unit disc $D_1$ with $D_R$, where
$b>0$ measures the distance between the two centers of $D_1$ and $D_R$. Let
$\Gamma=\partial Q$ be the boundary of $Q$, and $\Gamma_1$ be the circular
boundary component of $Q$ on the disk $D_1$, and $\Gamma_R$ be the circular
boundary component of $Q$ on the the disk $D_R$. Let $A$ and $B$ be the
points of intersection of $\Gamma_1$ and $\Gamma_R$, whom we will call the
{\it corner} points of $Q$. It is easy to see the following two extreme
cases: $Q(b,R)=D_1$ when $b\le R-1$, and $Q(b,R)=\emptyset$ when $b\ge 1+R$.
So we will assume $b\in (R-1,R+1)$ for the rest of this paper.

We first review some properties of periodic points of the billiard system on
$Q(b,R)$. It is easy to see that there is no fixed point, and exactly one
period 2 orbit colliding with both arcs\footnote{There are some other period
2 orbits which only collide with $\Gamma_1$. These orbits are parabolic.},
say $\mathcal{O}_2$, which moves along the segment passing through both
centers. The following result is well known, see \cite{Woj86} for example.
\begin{lem}
The orbit $\mathcal{O}_2$ is hyperbolic if $1<b<R$, is parabolic if $b=1$ or
$b=R$, and is elliptic if $b<1$ or $b>R$.
\end{lem}

It has been observed in \cite{CMZZ} that under the condition $b<1$ or $b>R$,
$\mathcal{O}_2$ is actually nonlinearly stable (see also \cite{OP05}). That
is, the orbit $\mathcal{O}_2$ is surrounded by some islands. Therefore, the
following is a necessary condition such that the billiard system on $Q(b,R)$
is hyperbolic.

\vskip.1in

\noindent\textbf{(A0)} The parameters $(b,R)$ satisfy $\max\{R-1,1\}< b< R$.

\vskip.1in

In this paper we prove that the billiard system on $Q(b,R)$ is completely
hyperbolic under the assumption \textbf{(A0)} and some general assumptions
\textbf{(A1)}--\textbf{(A3)}. As these assumptions are rather technical, we
will state them in Section \ref{assumption}.

\begin{thm}\label{hyper}
Let $Q(b,R)$ be an asymmetric lemon table satisfying the assumptions
\textbf{(A0)}--\textbf{(A3)}. Then the billiard system on $Q(b,R)$ is
hyperbolic.
\end{thm}

The proof of Theorem \ref{hyper} is given in Section \ref{provehyper}.

To provide more intuitions for these conditions,  we consider a special class
of asymmetric lemon billiards.  We first cut the unit disk $D_1$ by a chord
with end point $A$ and $B$, and let $\Gamma_{1}$ be the major arc of the unit
circle with end points $A,B$. Denoted by $Q_0$ the larger part of the disk
whose boundary contains $\Gamma_1$. By the classical defocusing mechanism,
the billiard on $Q_0$ is hyperbolic and ergodic. Now replace the chord with a
circular arc on the circle $D_R$, for some large radius $R$. Note that the
distance between the two centers is given by
$b=(R^2-|AB|^2/4)^{1/2}-(1-|AB|^2/4)^{1/2}$. The resulting table
$Q(R):=Q(b,R)$ can be viewed as a perturbation of $Q_0$. The following
theorem shows that the billiard system on $Q_0$, while being nonuniformly
hyperbolic, is robustly hyperbolic under suitable perturbations.

\begin{thm}\label{main}
Let $\Gamma_{1}$ be the major arc of the unit circle whose end points $A,B$
satisfy $|AB|<1$. Then there exists $R_\ast>1$ such that for
each $R\ge R_\ast$, the billiard system on the table $Q(R)$ with two
corners at $A,B$ is hyperbolic.
\end{thm}
The proof of Theorem \ref{main} is given in Section \ref{provemain}.

\begin{remark}
The hyperbolicity of the billiard system guarantees that a typical
(infinitesimal) wave front in the phase space grows exponentially fast along
the iterations of the billiard map. Therefore,  one can say that these
billiards in Theorem \ref{main} still demonstrate the defocusing mechanism.
However, the circle completing each of the boundary arcs of the table
$Q(b,R)$ {\it contains} the entire table. Therefore, the defocusing mechanism
can generate hyperbolicity even in the case when the separation condition is
strongly violated.
\end{remark}


The assumption $|AB|<1$ in Theorem \ref{main} is a purely technical one, and
it is used only once in the proof of Theorem \ref{main} to ensure $n^\ast\ge 6$
(see the definition of $n^\ast$ in \S \ref{provemain}). Clearly this
assumption $|AB|<1$ is stronger than the assumption that $\Gamma_1$ is a
major arc. We conjecture that as long as $\Gamma_1$ is a major arc, the
billiard system on $Q(R)$ is completely hyperbolic for any large enough
$R$.
\begin{conj}
Fix two points $A$ and $B$ on $\partial D_1$ such that $\Gamma_1$ is a major
arc. Then  the billiard system on $Q(b,R)$ is hyperbolic if the center of the
disc $D_R$ lies out side of the table.
\end{conj}

To ease a task of reading  we provide hereby a list of notations that we use
in this paper.

\begin{center}
{The List of Notations}
\end{center}
\begin{flushleft}
\begin{tabular}{ l p{3.9in} }
$Q(r,b,R)=D_r\cap D_R$ & the asymmetric lemon table as the intersection of
$D_r$ with $D_R$. We usually set $r=1$ and denote it by $Q(b,R)=Q(b,1,R)$. We also
denote $Q(R)=Q(b,R)$ if $b$ is determined by $R$.\\

$\Gamma=\partial Q(b,R)$ & the boundary of $Q(b,R)$, which consists of two
arcs: $\Gamma_1$ and $\Gamma_R$.\\

$\cM=\Gamma\times [-\pi/2,\pi/2]$ & the phase space with coordinate
$x=(s,\varphi)$,
which consists of $\cM_1$ and $\cM_R$.\\

$\cF$ & the billiard map on the phase space $\cM$ of $Q(b,R)$.\\

$\cS_1$ & the set of points in $\cM$ at where $\cF$ is not well defined or
not
smooth.\\

$\cB^{\pm}(V)$ &  the curvature of the orthogonal transversal of the beam of
lines generated by a tangent vector $V\in T_x\cM$ before and after
the reflection at $x$, respectively.\\

$d(x)=\rho\cdot \cos\varphi$ & the half of the chord cut out by the
trajectory of the billiard orbit in the disk $D_\rho$, where $\rho\in\{1,R\}$
is given by $x\in
\cM_\rho$.\\

$\cR(x)=-\frac{2}{d(x)}$ & the reflection parameter that measures the
increment of the curvature after reflection.\\

$\tau(x)$ &  the distance between the current position of $x$ with the next
reflection with  $\Gamma$.\\

$\chi^{\pm}(\cF,x)$ & the Lyapunov exponents of $\cF$ at the point $x$.
\end{tabular}

\begin{tabular}{ l p{3.9in} }
$\eta(x)$ & the number of successive reflections of $x$ on the arc
$\Gamma_\sigma$, where $\sigma$ is given by $x\in\cM_\sigma$.\\

$\hat \cM_1=\cM_1\backslash \cF\cM_1$ & the set of points
that first enter $\cM_1$. Similarly we define $\hat \cM_R$.\\

$M_n=\{x\in\hat\cM_1:\eta(x)=n\}$ & the set of points in $\hat \cM_1$ having
$n$ reflections on $\Gamma_1$ before hitting $\Gamma_R$.\\

$\hat F(x)=\cF^{j_0+j_1+2}x$ & the first return map of $\cF$ on the subset
$\hat\cM_1$, where $j_0$ is the number of reflections of $x$ on $\Gamma_1$,
and $j_1$ is the number of reflections of $x_1=\cF^{j_0+1}x$ on $\Gamma_R$.\\

$\hat\tau_k=\tau_k-j_k\hat d_k-j_{k+1}\hat d_{k+1}$  & a notation for short,
where $\tau_k=\tau(x_k)$,
$d_k=d(x_k)$, $\hat d_k=\frac{d_k}{j_k+1}$.\\

$\hat \cR_k=-\frac{2}{\hat d_k}$  &  a notation for short.\\

$M=\bigcup_{n\ge 0}\cF^{\lceil n/2\rceil}M_n$ & a subset of $\cM_1$. Compare
with the set $\hat \cM_1=\bigcup_{n\ge 0}M_n$.\\

$F(x)=\cF^{i_0+i_1+i_2+2}x$ & the first return map of $\cF$ on $M$, where
$i_0$ is the number of reflections of $x$ on $\Gamma_1$, $i_1$ is the number
of reflections of $x_1=\cF^{i_0+1}x$ on $\Gamma_R$, and $i_2$ is the number
of
reflections of $x_2=\cF^{i_1+1}x_1$ on $\Gamma_1$ before entering $M$.\\

$\bar \tau_k=\tau_k-i_k\hat d_k-d_{k+1}$  & a notation for short. Only $\bar
\tau_1$ is used in this paper.
\end{tabular}
\end{flushleft}

\section{Preliminaries for general convex billiards}

Let $Q\subset \mathbb{R}^2$ be a compact convex domain with piecewise smooth
boundary, $\cM$ be the space of unit vectors based at the boundary
$\Gamma:=\partial Q$ pointing inside of $Q$. The set $\cM$ is endowed with
the topology induced from the tangent space $TQ$. A point $x\in\cM$
represents the initial status of a particle, which moves along the ray
generated by $x$ and then makes an elastic reflection after hitting $\Gamma$.
Denote by $x_1\in\cM$ the new status of the particle right after this
reflection. The billiard map, denoted by $\cF$, maps each point $x\in\cM$ to
the point $x_1\in\cM$. Note that each point $x\in\cM$ has a natural
coordinate $x=(s, \varphi)$, where $s\in [0,|\Gamma|)$ is the arc-length
parameter of $\Gamma$ (oriented counterclockwise), and $\varphi\in
[-\pi/2,\pi/2]$ is the angle formed by the vector $x$ with the inner normal
direction of $\Gamma$ at the base point of $x$. In particular, the phase
space $\cM$ can be identified with a cylinder $\Gamma\times[-\pi/2,\pi/2]$.
The billiard map preserves a smooth probability measure $\mu$ on $\cM$, where
$d\mu=(2|\Gamma|)^{-1}\cdot\cos\varphi\, ds\, d\varphi$.

For our lemon table  $Q=Q(b,R)$, the boundary $\Gamma=\partial Q$ consists of
two parts $\Gamma_1$ and $\Gamma_R$. Collision vector starting from a corner
point at $A$ or $B$ has $s$-coordinate $s=0$ or $s=|\Gamma_1|$, respectively.
Then we can  view $\cM$ as the union of two closed rectangles:
$$\cM_{1}:=\{(s,\varphi)\in\cM\,:\,0\le s \le |\Gamma_1|\}\,\,\,\,\,\,\text{ and
}\,\,\,\,\,\cM_{R}:=\{(s,\varphi)\in\cM\,:\,|\Gamma_1|\le s\le |\Gamma|\}.$$ For
any point $x=(s,\varphi)\in \cM$,  we define $d(x)=\cos\varphi$ if
$x\in\cM_1$, and $d(x)=R\cos\varphi$ if $x\in\cM_R$. Geometrically, the
quantity $2d(x)$ is the length of the chord in the complete disk ($D_1$ or
$D_R$) decided by the trajectory of $x$.

Let $\cS_0=\{(s,\varphi)\in\cM: s=0\text{ or  } s=|\Gamma_1|\}$ be the set of
post-reflection vectors $x\in\cM$ that pass through one of the corners $A$ or
$B$. We define $\cS_1=\cS_0\cup \cF^{-1}\cS_0$ as the set of points on which
$\cF$ is not well-defined. Note that $\cF^{-1}\cS_0$  consists of $4$
monotone curves $\varphi=\varphi_i(s)$ ($1\le i\le 4$) in $\cM$. Moreover, we
define $\cS_{-1}:=\cS_0\cup \cF\cS_0$. The set $\cS_{\pm 1}$ is called the
{\it singular set} of the billiard map $\cF^{\pm 1}$.

\begin{figure}[htb]
\centering
\begin{overpic}[width=85mm]{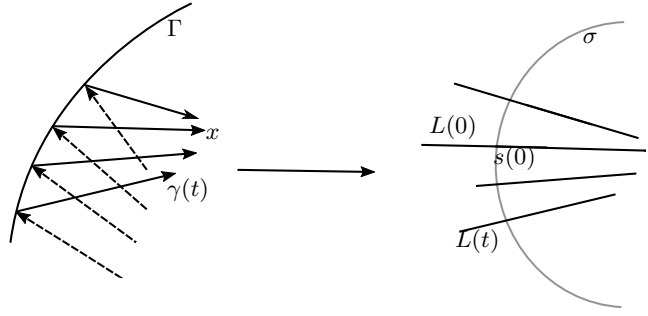}
\put(25,43){\small$\Gamma$}
\put(31,27){\small$x$}
\put(25,18){\small$\gamma(t)$}
\put(90,42){\small$\sigma$}
\put(66,28){\small$L(0)$}
\put(76,23){\small$s(0)$}
\put(70,10){\small$L(t)$}
\end{overpic}
\caption{\small{A bundle of lines generated by $\gamma$, and the cross-section $\sigma$.}}
\label{wavefront}
\end{figure}

A way to understand chaotic billiards lies in the study of infinitesimal
families of trajectories.  More precisely, let $x\in \cM\setminus\cS_1$,
$V\in T_x\cM$, $\gamma: (-\eps_0,\eps_0)\to\cM$, $t
\mapsto\gamma(t)=(s(t),\varphi(t))$ be a smooth curve for some $\eps_0>0$,
such that $\gamma(0)=x$ and $\gamma'(0)=V$. Clearly the choice of such a
smooth curve is not unique. Each point in the phase space $\cM\subset TQ$ is
a unit vector on the billiard table $Q$. Let $L(t)$ be the line that passes
through the vector $\gamma(t)$, see Fig. \ref{wavefront}. Putting these lines
together, we get a beam of post-reflection lines, say $\cW^{+}$, generated by
the path $\gamma$. Let $\sigma$ be the orthogonal cross-section of this
bundle passing through the point $s(0)\in\Gamma$. Then the post-reflection
curvature of the tangent vector $V$, denoted by $\cB^+(V)$, is defined as the
curvature of $\sigma$ at the point $s(0)$. Similarly we define the
pre-reflection curvature $\cB^{-}(V)$ (using the beam of dashed lines in Fig.
\ref{wavefront}).

Note that $\cB^{\pm}(V)$ depend only on $V$, and are independent of the
choices of curves tangent to $V$. These two quantities are related by the
equation
 \beq\label{fpm}
\cB^+(V)-\cB^-(V)=\cR(x),
 \eeq
where $\cR(x):=-2/d(x)$ is the {\it{reflection parameter}} introduced in
\cite{Si70}, see also \cite[\S 3.8]{CM06}. In fact, \eqref{fpm} is the
well-known {\it Mirror Equation} in geometric optics. Note that $\cR(x)>0$ on
dispersing components and $\cR(x)<0$ on focusing components  of the boundary
$\partial Q$. Since we mainly use $\cB^-(V)$ in this paper, we drop the minus
sign, simply denote it by $\cB(V)=\cB^-(V)$.

Let $\tau(x)$ be the distance from the current position of $x$ to the next
reflection with $\Gamma$. According to (\ref{fpm}), one gets the evolution
equation for the curvatures of the pre-reflection wavefronts of $V$ and its
image $V_1=D\cF(V)$ at $\cF x$:
 \beq\label{ff1}
\cB_1(V):=\cB(V_1)=\frac{1}{\tau(x)+\dfrac{1}{\cR(x)+\cB(V)}}.
 \eeq
More generally, let $x\in \cM\setminus \cS_1$ be a point with $\cF^k x\notin
\cS_1$ for all $1\le k\le n$, $V\in T_x\cM$ be a nonzero tangent vector, and
$V_{n}=D\cF^{n} V$ be its forward iterations. Then by iterating the formula
\eqref{ff1}, we get
\begin{align}
\cB(V_{n})&=\frac{1}{\tau(x_{n-1})+\dfrac{1}{\cR(x_{n-1})+\dfrac{1}{\tau(x_{n-2})
+\dfrac{1}{\cR(x_{n-2})+\dfrac{1}{\ddots+\dfrac{1}{\tau(x)+\dfrac{1}{\cR(x)+\cB(V)}}}}}}}
\label{ff2}
\end{align}
with $x_{k}=\cF^{k}x$. See also \cite[\S 3.8]{CM06} for Eq. \eqref{ff1} and
\eqref{ff2}.

\vskip.3cm

For convenience, we introduce the standard notations for continued fractions
\cite{Kh64}. In the following, we will denote $[a]:=\frac{1}{a}$. The reader
should not be confused by the integral part of $a$, which is never used in
this paper\footnote{We use the ceiling function $\lceil
t\rceil=\min\{n\in\mathbb{Z}: n\ge t\}$ in \S \ref{induced}, and the floor
function $\lfloor t\rfloor=\max\{n\in\mathbb{Z}: n\le t\}$ in \S
\ref{provemain}.}.
\begin{definition}
Let $a_n$, $n\ge0$ be a sequence of real numbers. The finite continued
fraction $[a_1,a_2,\cdots, a_n]$ is defined inductively by:
\begin{align*}
[a_1]=\frac{1}{a_1},\, [a_1,a_2]&=\frac{1}{a_1+[a_2]},\, \cdots,\,
[a_1,a_2,\cdots, a_n]=\frac{1}{a_1+[a_2,\cdots, a_n]}.
\end{align*}
Moreover, we denote $[a_0;a_1,a_2,\cdots, a_n]=a_0+[a_1,a_2,\cdots, a_n]$.
\end{definition}

Using this notation, we see that the evolution \eqref{ff2} of the curvatures
of $V_n=D\cF^n(V)$ can be re-written as
 \beq\label{ffn}
 \cB(V_n)=[\tau(x_{n-1}),\cR(x_{n-1}),\tau(x_{n-2}),
 \cR(x_{n-2}),\cdots,\tau(x),\cR(x)+\cB(V)].
 \eeq
Note that Eq. \eqref{ffn} is a recursive formula and hence can be extended
{\it formally} to an infinite continued fraction.

We will need the following basic properties of continued fractions to perform
some reductions. Let $x=[a_1,\cdots,a_n]$ be a finite continued fraction.
Then we can combine two finite continued fractions in the following ways:
\begin{align}
[b_1,\cdots,b_n+x]&=[b_1,\cdots,b_n,a_1,\cdots,a_n],\label{combine}\\
[b_1,\cdots,b_n,x]&=[b_1,\cdots,b_n+a_1,a_2,\cdots,a_n],\label{combine2}\\
[b_1,\cdots,b_n,0,a_1,\cdots,a_n]&=[b_1,\cdots,b_n+a_1,a_2,\cdots,a_n].\label{add0}
\end{align}

\begin{pro}\label{cont}
Suppose $a,b,c$ are real numbers such that $B:=a+c+abc\neq 0$. Then the
relation
\begin{align}
[\cdots,x,a,b,c,y,\cdots]&=[\cdots,x+A,B,C+y,\cdots]\label{ABC}
\end{align}
holds for any finite or infinite continued fractions, where $A=\frac{bc}{B}$
and $C=\frac{ab}{B}$.
\end{pro}

Let $Q$ be a bounded domain with piecewise smooth boundary, $\cF$ be the
billiard map on the phase space $\cM$ over $Q$. Then the limit $\ds
\chi^+(\cF,x)=\lim_{n\to\infty}\frac{1}{n}\log\|D_x\cF^n\|$, whenever it
exists, is said to be a {\it Lyapunov exponent} of the billiard map $\cF$ at
the point $x$. Since $\cF$ preserves the smooth measure $\mu$, the other
Lyapunov exponent at $x$ is given by $\chi^-(\cF,x)=-\chi^+(\cF,x)$. Then the
point $x$ is said to be hyperbolic, if $\chi^+(\cF,x)>0$. Moreover, the
billiard map $\cF$ is said to be (completely) {\it hyperbolic}, if
$\mu$-almost every point $x\in \cM$ is a hyperbolic point. By Oseledets {\it
Multiplicative Ergodic Theorem}, we know that $\chi^+(\cF,x)$ exists for
$\mu$-a.e. $x\in \cM$, and there exists a measurable splitting
$T_x\cM=E^u_x\oplus E^s_x$ over the set of hyperbolic points, see
\cite{CM06}.

It is well known that the hyperbolicity of a billiard map is related to the
{\it convergence} of the continued fraction given in Eq. \eqref{ffn} as
$n\to\infty$. In particular, the following proposition reveals the relations
between them. See \cite{Bu74A,CM06, Si70}.
\begin{pro}
Let $x\in \cM$ be a hyperbolic point of the billiard map $\cF$. Then the
curvature $\cB^u(x):=\cB(V^u_x)$ of a unit vector $V^u_x\in E^u_x$ is given
by the following infinite continued fraction:
$$\cB^u(x)=[\tau(x_{-1}),\cR(x_{-1}),\tau(x_{-2}),\cR(x_{-2}),\cdots,
\tau(x_{-n}),\cR(x_{-n}),\cdots].$$
\end{pro}

Finally we recall an invariant property for consecutive reflections on
focusing boundary components by comparing the curvatures of the iterates of
different tangent vectors. Given two distinct points $a$ and $b$ on the unit
circle $\mathbb{S}^1$, denote by $(a,b)$ the interval from $a$ to $b$
counterclockwise. Given three distinct points $a,b,c$ on $\mathbb{S}^1$,
denote by $a\prec b \prec c$ if $b\in (a,c)$. Endow
$\mathbb{R}\cup\{\infty\}\simeq \mathbb{S}^1$ with the relative position
notation $\prec$ on $\mathbb{S}^1$.
\begin{pro}[\cite{Do91}]\label{order}
Let $X,Y,Z\in T_x\cM$ be three tangent vectors at $x\in \cM$ satisfying
$\cB(X)\prec\cB(Y)\prec\cB(Z)$. Then for each $n\in\mathbb{Z}$, the iterates
$D\cF^nX$, $D\cF^nY$ and $D\cF^nZ$ satisfy
$$\cB(D\cF^nX)\prec\cB(D\cF^nY)\prec\cB(D\cF^nZ).$$
\end{pro}

\section{Continued fractions for asymmetric lemon billiards}\label{cfrac}

In this section we construct two induced maps of the billiard system
$(\cM,\cF)$ on two different but closely related subsets of the phase space
$\cM$, and then study the evolutions of continued fractions of the curvatures
$\cB(V)$ under these induced maps. Let $Q(b,R)$ be an asymmetric lemon table
obtained as the intersection of a disk of radius 1 with a disk of radius
$R>1$, $\Gamma=\partial Q$, $\cM=\Gamma\times [-\pi/2,\pi/2]$ be the phase
space of the billiard map on $Q$. Note that $\cM$ consists of two parts:
$\cM_1:=\Gamma_1\times[-\pi/2,\pi/2]$ and
$\cM_R=\Gamma_R\times[-\pi/2,\pi/2]$, the sets of points in $\cM$ based on
the arc $\Gamma_{1}$ and $\Gamma_{R}$, respectively. Assume that $\Gamma_1$
is a major arc.

For $\sigma\in \{1,R\}$, $x\in \cM_{\sigma}\setminus \cS_1$, let $\eta(x)$ be
the number of successive reflections of $x$ on the arc $\Gamma_{\sigma}$.
That is,
 \beq
 \eta(x)=\sup\{n\geq 0\,:\, \cF^{k} x\in \cM_{\sigma} \text{ for all } k=0,\cdots,n\}.
 \eeq
For example, $\eta(x)=0$ if $\cF x\notin \cM_{\sigma}$, and $\eta(x)=\infty$
if $\cF^kx\in \cM_\sigma$ for all $k\ge 0$. Let $N=\{x\in
\cM_1\,:\,\eta(x)=\infty \}$. One can easily check that each point $x\in N$
is either periodic or belongs to the boundary $\{(s,\varphi)\in \cM_1\,:\,
\varphi=\pm \pi/2\}$. In particular, $N$ is a null set with $\mu(N)=0$.

Let $\hat \cM_{1}:=\{x\in \cM_{1}:\cF^{-1}x\notin \cM_{1}\}$ be the set of
points first entering $\cM_1$. Similarly we define $\hat \cM_{R}$. The
restriction of $\eta$ on $\hat \cM_{1}$ induces a measurable partition of
$\hat \cM_{1}$, whose cells are given by $M_n:=\eta^{-1}\{n\}\cap\hat \cM_1$
for all $n\geq 0$. Each cell $M_n$ contains all first reflection vectors on
the arc $\Gamma_1$ that will experience exactly $n$ reflections on $\Gamma_1$
before hitting $\Gamma_R$. Then it is easy to check that
\begin{equation}\label{decomp}
\cM_1=N\cup\bigcup_{n\ge 0}\bigcup_{0\le k\le n}\cF^k M_n.
\end{equation}

\subsection{The first induced map of $\cF$ on $\cM_1$}
Let $x_0\in \hat \cM_1$, and $j_0=\eta(x_0)$ be the numbers of successive
reflections of $x_0$ on $\Gamma_1$. Similarly, we denote $x_1=\cF^{j_0+1}
x_0$, and $j_1=\eta(x_1)$. Then the first return map $\hat F$ of $\cF$ on
$\hat\cM_1$ is given by
$$\hat Fx:=\cF^{j_0+j_1+2}x.$$
Note that the similar induced systems appeared in many references about
billiards with convex boundary components, see \cite{CM06,CZ05, M04}.  In the
systems considered in these references, the induced systems were shown to be
(uniformly) hyperbolic.  However, for our billiard systems on $Q(b,R)$, it is
rather difficult to prove the hyperbolicity for this type of the induced map.
Thus we introduce a new induced map in the next subsection. To make a
comparison, we next investigate the properties of  the induced map
$(\hat\cM_1, \hat F)$.

To simplify the notations, we denote by $\tau_0:=\tau(\cF^{j_0} x_0)$ the
length of the free path of $\cF^{j_0} x_0$, and by $\tau_1:=\tau(\cF^{j_1}
x_1)$ the length of the free path of $\cF^{j_1} x_1$. Moreover, let
$d_k=d(x_k)$, $\hat d_k=\frac{d_k}{j_k+1}$, $\hat\cR_k=-2/\hat d_k$, $\hat
\tau_k=\tau_k-j_k\hat d_k-j_{k+1}\hat d_{k+1}$, for $k=0,1$. Note that $\hat
d_k=d_k$ and $\hat\cR_k=\cR_k$ if $j_k=0$, and $\hat\tau_k=\tau_k$ if
$j_k=j_{k+1}=0$.

\vskip.1in

Using the relations in Proposition \ref{cont}, we can reduce the long
continuous fraction to a shorter one:
\begin{lem}\label{reduce1}
Let $x\in\hat \cM_1$, $V\in T_x\cM$ and $\hat V_1=D\hat F(V)$. Then $\cB(\hat
V_1)$ is given by the continued fraction:
\begin{align}\label{fract}
\cB(\hat V_1)&=[\tau_1-j_1\hat d_1,\hat\cR_1,
\hat\tau_0,\hat\cR_0,-j_0\hat d_0, \cB(V)].
\end{align}
\end{lem}
\begin{remark}
Note that in the case $j_0=0$ and $j_1=0$, $\hat F x =\cF^2 x$, and the
relation \eqref{fract} reduces to the formula \eqref{ffn} with $n=2$:
$\cB(\hat V_1)=[\tau_1,\cR_1, \tau_0,\cR_0+\cB(V)]$.
\end{remark}
\begin{proof}
Suppose a point $x\in \cM$ have $m$ consecutive reflections on a circular arc
$\Gamma_\sigma$, where $\sigma\in\{1,R\}$. That is, $\cF^{i}x\in
\cM_{\sigma}$, for $i=0, \cdots, m$. In this case we always have
$$d(\cF^{i}x)= d(x),\, \cR(\cF^{i}x)=-2/d(x),\,  0\le i\le m,
\text{ and } \tau(\cF^{i}x)=2d(x), \, 0\le i< m.$$
Then for any $V\in
T_{x}\cM_{\sigma}$,
\begin{align}
\cB(D\cF^{m+1} V)&=[\tau(\cF^{m}x),\underbrace{\cR(x),
2 d(x), \cR(x),\cdots,2 d(x)}_\text{$m$ times },\cR(x)+\cB(V)]\nonumber\\
&=[\tau(\cF^{m}x),\cR(x)/2,-2m\cdot d(x),\cR(x)/2+\cB(V)],\label{redu0}
\end{align}
see \cite[\S 8.7]{CM06}. Applying this reduction process for each of the two
reflection series on $\Gamma_1$ and $\Gamma_R$ respectively, we see that
\begin{align}
\cB(\hat V_1)=[&\tau_1,\cR_1/2,-2j_1d_1,\cR_1/2,
\tau_0,\cR_0/2,-2j_0d_0,\cR_0/2+\cB(V)].\label{reduction1}
\end{align}
Now we rewrite the last segment in \eqref{reduction1} as
$[\cdots,\cR_0/2+\cB(V)]=[\cdots,\cR_0/2,0,\cB(V)]$ by Eq. \eqref{add0}. Then
applying Eq. \eqref{ABC} to the segment $(a,b,c)=(\cR_0/2,-2j_0d_0,\cR_0/2)$
in Eq. \eqref{reduction1}, we get
\begin{align*}
B_0&:=a+c+abc=\cR_0-(\cR_0/2)^2\cdot 2j_0d_0
=-\frac{2+2j_0}{d_0}=-\frac{2}{\hat d_0}=\hat\cR_0,\\
A_0&:=\frac{bc}{B_0}=-2j_0d_0\cdot \cR_0/2\cdot(-\frac{\hat d_0}{2})=-j_0\cdot\hat d_0,\\
C_0&:=\frac{ab}{B_0}=-2j_0d_0\cdot \cR_0/2\cdot(-\frac{\hat d_0}{2})=-j_0\cdot\hat d_0.
\end{align*}
Putting them together with \eqref{reduction1}, we have
\begin{align}
\cB(\hat V_1)&=
[\tau_1,\cR_1/2,-2j_1d_1,\cR_1/2,
\tau_0,\cR_0/2,-2j_0d_0,\cR_0/2,0,\cB(V)]\nonumber\\
&=[\tau_1,\cR_1/2,-2j_1d_1,\cR_1/2,\tau_0+A_0,B_0,0+C_0,\cB(V)]\nonumber\\
&=[\tau_1,\cR_1/2,-2j_1d_1,\cR_1/2,\tau_0-j_0\cdot\hat d_0,
\hat\cR_0,-j_0\cdot\hat d_0,\cB(V)].\label{inter1}
\end{align}

Similarly we can apply Eq. \eqref{ABC} to the segment
$(\cR_1/2,-2j_1d_1,\cR_1/2)$, and get $B_1=\hat \cR_1$,
$A_1=C_1=-j_1\cdot\hat d_1$. Then we can continue the computation from
\eqref{inter1} and get
\begin{align*}
\cB(\hat V_1)&=[\tau_1,\cR_1/2,-2j_1d_1,\cR_1/2,
\tau_0-j_0\cdot\hat d_0,\hat\cR_0,-j_0\cdot\hat d_0,\cB(V)]\\
&=[\tau_1+A_1,B_1,\tau_0+C_1-j_0\cdot\hat d_0,\hat\cR_0,-j_0\cdot\hat d_0,\cB(V)]\\
&=[\tau_1-j_1\hat d_1,\hat\cR_1,
\hat\tau_0,\hat\cR_0,-j_0\hat d_0, \cB(V)],
\end{align*}
where $\hat d_k=\frac{d_k}{j_k+1}$, $\hat\cR_k=-2/\hat d_k$ for $k=0,1$, and
$\hat \tau_0=\tau_0-j_0\hat d_0-j_{1}\hat d_{1}$. This completes the proof.
\end{proof}

It is clear that the formula in Eq. \eqref{fract} is recursive. For example,
let $\hat V_2=D\hat F^2(V)$, $\tau_2=\tau_0(\hat Fx)$ and $\tau_3=\tau_1(\hat
Fx)$ be the lengths of the free paths for $\hat Fx$. Then we have
\[\cB(\hat
V_2) =[\tau_3-j_3\hat d_3,\hat\cR_3, \hat\tau_2,\hat\cR_2, \hat\tau_1,\hat\cR_1,
\hat\tau_0,\hat\cR_0,-j_0\hat d_0, \cB(V)].\] %
More generally, using the backward iterates $x_{-n}=\hat F^{-n}x$ of $x$, and
the related notations (for example, $d_{-2n}=d_0(x_{-n})$, and
$d_{1-2n}=d_1(x_{-n})$), we get a formal continued fraction
\begin{equation}\label{tradition}
[\tau_0-j_0\hat d_0, \hat\cR_0,\hat\tau_{-1}, \hat\cR_{-1},
\hat\tau_{-2},\cdots, \hat\cR_{1-2n},\hat\tau_{-2n},\hat\cR_{-2n},\hat\tau_{-2n-1},\cdots]
\end{equation}
where $\hat d_k=\frac{d_k}{j_k+1}$, $\hat\cR_k=-2/\hat d_i$, $\hat
\tau_k=\tau_k-j_k\hat d_k-j_{k+1}\hat d_{k+1}$, for each $k\le -1$.

\begin{remark}
In the dispersing billiard case, each entry of the continued fraction
\eqref{ffn} is positive. Then Seidel--Stern Theorem (see \cite{Kh64}) implies
that the limit of \eqref{ffn} (as $n\to\infty$) always exists, since the
total time $\tau_0+\cdots+\tau_n\to\infty$. For the reduced continued
fraction \eqref{tradition}, it is clear that $\hat \cR_{-n}<0$ for all $n\ge
0$. Moreover, $\hat\cR_{-2n}\le\cR(x_{-n})\le -2$ for each $n\ge 1$, since
the radius of the small disk is set to $r=1$. Therefore, $\sum_n
\hat\cR_{-n}(x)$ always diverges. However, Seidel--Stern Theorem is not
applicable to determine the convergence of \eqref{tradition}, since the terms
$\hat \tau_k$ have no definite sign. This is the reason that we need to
introduce a new return map $(M,F)$ instead of using $(\hat M_1,\hat F)$ to
investigate the hyperbolicity.
\end{remark}

\subsection{New induced map and its analysis}\label{induced}
Denote by $\lceil t\rceil$ the smallest integer larger than or equal to the
real number $t$. We consider a new subset, which consists of ``middle"
sliding reflections on $\Gamma_1$. More precisely, let
 \beq
M:=\bigcup_{n\ge0}\cF^{\lceil n/2\rceil}M_n= M_0\cup \cF M_1\cup \cF
M_2\cup\cdots\cup \cF^{\lceil n/2\rceil}M_n\cup\cdots.
 \eeq
Let $F$ be the first return map of $\cF$ with respect to $M$. Clearly the
induced map $F:M\to M$ is measurable and preserves the conditional measure
$\mu_M$ of $\mu$ on $M$, which is given by $\ds \mu_M(A)=\mu(A)/\mu(M)$, for
any Borel measurable set $A\subset M$.

For each $x\in M$,  we introduce the following notations:
\begin{enumerate}
\setlength{\itemsep}{4pt}
\item let $i_0=\eta(x)\ge0$ be the number of forward reflections of $x$
    on $\Gamma_1$, $\tau_0:=\tau(\cF^{i_0}x)$
 be the distance between the last reflection on $\Gamma_1$ and the first
 reflection on $\Gamma_R$. Let $d_0:=d(x)$ and $\cR_0:=\cR(x)$, which
 stay the same along this series of reflections on $\Gamma_1$;

\item let $i_1=\eta(x_1)\ge0$ be the number of reflections of
    $x_1=\cF^{i_0+1}x$ on $\Gamma_R$, $\tau_1:=\tau(\cF^{i_1}x_1)$ be the
    distance between the last reflection on $\Gamma_R$ and the next
    reflection on $\Gamma_1$. Let $d_1=d(x_1)$, and $\cR_1=\cR(x_1)$,
    which stay the same along this series of reflections on $\Gamma_R$;

\item let $i_2=\lceil\eta(x_2)/2\rceil\ge0$ be the number\footnote{ Our
    choice of $i_2=\lceil\eta(x_2)/2\rceil$ in Item (3), instead of using
    $\eta(x_2)$, is due to the fact that $M$ is the union of the sets
    $\cF^{\lceil n/2\rceil}M_n$, $n\ge 0$.} of reflections of
    $x_2=\cF^{i_1+1}x_1$ on $\Gamma_1$ till the return to $M$. Let
    $d_2=d(x_2)$, $\cR_2=\cR(x_2)$, which stay the same along this series
    of reflections on $\Gamma_1$.
\end{enumerate}
Then the first return map $F$ on $M$ is given explicitly by
$Fx=\cF^{i_0+i_1+i_2+2}x$. Note that $x_2=Fx$ and $d_2=d(Fx)$ (in above
notations).

The following result is the analog of Lemma \ref{reduce1} on the reduction of
continued fractions for the new induced return map $F$:
\begin{lem}\label{reduce2}
Let $x\in M$, $V\in T_x\cM$ and $V_1=DF(V)$. Then $\cB(V_1)$ is given by the
continued fraction:
\begin{align}\label{fract2}
\cB(V_1)&=[d_2,\frac{2i_2}{d_2},\bar\tau_1,\hat\cR_1,
\hat\tau_0,\hat\cR_0,-i_0\hat d_0,\cB(V)],
\end{align}
where $\hat d_k=\frac{d_k}{j_k+1}$, $\hat\cR_k=-2/\hat d_k$,
$\hat\tau_0=\tau_0-i_0\hat d_0-i_{1}\hat d_{1}$, and
$\bar\tau_1=\tau_1-i_1\hat d_1-d_{2}$.
\end{lem}
Note that \eqref{fract2} may not be as pretty as \eqref{fract}. It involves
three types of quantities: the original type ($i_0$, $i_2$ and $d_2$), the
first variation ($\hat R_k$, $\hat \tau_0$ and $\hat d_0$), and the second
variation $\bar \tau_1$. The quantity $\hat d_2$ appears only in the
intermediate steps of the proof, and does not appear in the final formula
\eqref{fract2}.
\begin{remark}\label{remark-i2}
Note that in the case $i_2=0$, \eqref{fract2} reduces to \eqref{fract}:
\begin{align*}
\cB(V_1)&=[d_2,0,\bar\tau_1,\hat\cR_1, \hat\tau_0,\hat\cR_0,-i_0\hat d_0,\cB(V)]
=[d_2+\bar\tau_1,\hat\cR_1, \hat\tau_0,\hat\cR_0,-i_0\hat d_0,\cB(V)]\\
&=[\tau_1-i_1\hat d_1,\hat\cR_1, \hat\tau_0,\hat\cR_0,-i_0\hat d_0,\cB(V)].
\end{align*}
\end{remark}
\begin{proof}
We first consider an intermediate step. That is, let $\hat
V_1=D\cF^{i_0+i_1+2}(V)$ and $V_1=D\cF^{i_2}(\hat V_1)$. Applying the
reduction process \eqref{redu0} for each of the series of reflection of
lengths $(i_0,i_1)$ (as in the proof of Lemma \ref{reduce1}), we get that
\begin{align*}
\cB(\hat V_1)=[\tau_1-i_1\hat d_1,\hat\cR_1, \hat\tau_0,\hat\cR_0,-i_0\hat d_0,\cB(V)].
\end{align*}
This completes the proof of \eqref{fract2} when $i_2=0$ (see Remark
\ref{remark-i2}). In the following we assume $i_2\ge 1$. In this case we have
$\cB(V_1)=[\underbrace{2 d_2, \cR_2,\cdots,2d_2,\cR_2}_\text{$i_2$
times}+\cB(\hat V_1)]$. To apply the reduction \eqref{redu0}, we need to
consider
\begin{equation}\label{supply}
\frac{1}{\cB(V_1)+\cR_2}=[\underbrace{\cR_2,2 d_2,
\cdots,\cR_2,2d_2}_\text{$i_2$ times},\cR_2+\cB(\hat V_1)].
\end{equation}
Then we apply the reduction \eqref{redu0} to Eq. \eqref{supply} and get $\ds
\frac{1}{\cB(V_1)+\cR_2}=[\cR_2/2,-i_2\hat d_2,\cR_2/2+\cB(\hat V_1)]$, which
is equivalent to
\begin{equation}\label{inverse}
\cB(V_1)=-\cR_2/2+[-i_2\hat d_2,\cR_2/2+\cB(\hat V_1)]
=[0,-\cR_2/2,-i_2\hat d_2,\cR_2/2,0,\cB(\hat V_1)].
\end{equation}
Applying the relation
\eqref{ABC} to the segment $(a,b,c)=(-\cR_2/2,-2i_2d_2,\cR_2/2)$, we get
\begin{align*}
B&:=a+c+abc=0+(\cR_2/2)^2\cdot 2i_2d_2=\frac{2i_2}{d_2},\\
A&:=\frac{bc}{B}=-2i_2d_2\cdot \cR_2/2\cdot\frac{d_2}{2i_2}=d_2,\\
C&:=\frac{ab}{B}=2i_2d_2\cdot \cR_2/2\cdot\frac{d_2}{2i_2}=-d_2.
\end{align*}
Putting them together with Eq. \eqref{inverse}, we have
\begin{align*}
\cB(V_1)=[d_2,\frac{2i_2}{d_2},-d_2,\cB(\hat V_1)]
&=[d_2,\frac{2i_2}{d_2},\bar\tau_1,\hat\cR_1,
\hat\tau_0,\hat\cR_0,-i_0\hat d_0,\cB(V)],
\end{align*}
where $\bar\tau_1=\tau_1-i_1\hat d_1-d_{2}$ follows from \eqref{combine2}.
This completes the proof of  \eqref{fract2}.
\end{proof}

\section{Hyperbolicity of asymmetric lemon billiards}\label{Sec:4}

In this section we first list several  general sufficient conditions that
ensure the hyperbolicity of the asymmetric lemon-type billiards (see the
statement below and the proof of Theorem \ref{hyper}), then we verify these
conditions for a set of asymmetric lemon tables. Let $Q(b,R)$ be an
asymmetric lemon table satisfying (A0), that is, $\max\{R-1,1\}<b<R$. Let $M$
be the subset introduced in \S\ref{induced}. We divide the set $M$ into three
disjoint regions $X_k$, $k=0,1,2$, which are given by
\begin{enumerate}
\item[(a)] $X_0=\{x\in M\,:\, i_1(x)\ge 1\}$;

\item[(b)] $X_1=\{x\in M\,:\, i_1(x)=0, \text{ and } i_2(x)\ge1\}$;

\item[(c)] $X_2=\{x\in M\,:\, i_1(x)=i_2(x)=0\}$.
\end{enumerate}

We first make a simple observation:
\begin{lem}
For each $x$ with $i_1(x)=0$, one has  $\tau_0+\tau_1>d_0+d_2$.
\end{lem}

\begin{proof}
Suppose $i_1(x)=0$, and $p_1$ be the reflection point of $x_1$ on $\Gamma_R$.
Then we take the union of $Q(b,R)$ with its mirror, say $Q^\ast(b,R)$, along
the tangent line $L$ of $\Gamma_R$ at $p_1$, and extend the pre-collision
path of $x_1$ beyond the point $p_1$, which will intersect $\partial
Q^\ast(b,R)$ at the mirror point of the reflection point $p_2$ of $\cF x_1$,
say $p^\ast_2$ (see Fig .\ref{mirror}). Clearly the distance
$|p_2^\ast-p_1|=|p_2-p_1|=\tau_1$. By the assumption that $b>R-1$, one can
see that the tangent line $L$ cuts out a major arc on $\partial D_1$ (clearly
larger than $\Gamma_1$), and the point $p_2^\ast$ lies outside of the unit
disk $D_1$. Therefore, $\tau_0+\tau_1>2d_0$. Similarly, we have
$\tau_0+\tau_1>2d_2$. Putting them together, we get $\tau_0+\tau_1>d_0+d_2$.
\end{proof}

\begin{figure}[htb]
\centering
\begin{overpic}[width=60mm]{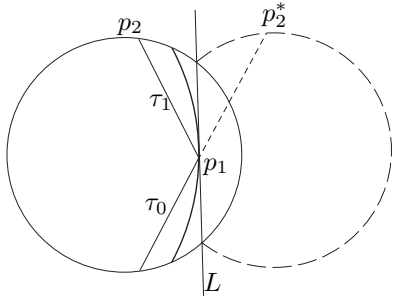}
\put(51,32){$p_1$}
\put(32,63){$p_2$}
\put(64,65){$p_2^\ast$}
\put(51,5){$L$}
\put(39,46){$\tau_1$}
\put(38,23){$\tau_0$}
\end{overpic}
\caption{The case with $i_1(x)=0$: there is  only one reflection on $\Gamma_R$.}
\label{mirror}
\end{figure}

\subsection{Main Assumptions and their analysis}\label{assumption}
In the following we list the assumptions on $X_i$'s.

\noindent(\textbf{A1}) For $x\in X_0$: $i_0\ge1$, $i_2\ge1$ and
\begin{equation}\label{assume1}
\tau_0<(1-\frac{1}{2(1+i_0)}) d_0+\frac{i_1}{1+i_1}d_1,\quad
\tau_1<(1-\frac{1}{2i_2})d_2+\frac{i_1}{1+i_1}d_1;
\end{equation}

\noindent(\textbf{A2}) For $x\in X_1$: $\frac{d_0}{1+i_0}< d_1$ and
 $\tau_0+\tau_1<(1-\frac{1}{2(1+i_0)})d_0+ d_1+(1-\frac{1}{2i_2})d_2$;

\noindent(\textbf{A3}) For $x\in X_2$: $\tau_0\le d_1/2$.

\vskip.3cm

To prove Theorem \ref{hyper}, it suffices to verify hyperbolicity of the
first return map $F:M\to M$, obtained by restricting $\cF$ on $M$. For each
$x\in M$, $V\in T_x\cM$, we let $\cB(V)=\cB^-(V)$ be the pre-reflection
curvature of $V$. Note that $\cB(V)$ determines $V$ uniquely up to a scalar.
Let $V^d_x\in T_x\cM$ be the unit vector corresponding to the incoming beam
with curvature $\cB(V^d_x)=1/d(x)$, and $V^{p}_x\in T_x\cM$ be the unit
vector corresponding to the parallel incoming beam $\cB(V^p_x)=0$,
respectively.
\begin{pro}\label{suffness}
Let $x\in M$, $i_k$, $d_k$ and $\hat d_k$, $k=0,1,2$ be the corresponding
quantities of $x$ given in \S \ref{induced}, and $Fx=\cF^{i_0+i_1+i_2+2}x$ be
the first return map of $\cF$ on $M$. Then
we have

\noindent $\mathrm{(I)}$. $0<\cB(DF(V_x^d))<1/d_2$ if one of the
following conditions holds:
 \begin{enumerate}
\item[(\textbf{D1})] $\tau_1-i_1\hat d_1-d_{2}+[-\frac{2}{\hat
    d_1},\tau_0-d_0-i_{1}\hat d_{1}]>0$,

\item[(\textbf{F1})]  $\tau_1-i_1\hat d_1-d_{2}+[-\frac{2}{\hat
    d_1},\tau_0-d_0-i_{1}\hat d_{1}]<-\frac{d_2}{2i_2}$.
 \end{enumerate}

\vskip.2cm

\noindent $\mathrm{(II)}$. $0<\cB(DF(V_x^p))<1/d_2$ if one of the following
conditions holds:
 \begin{enumerate}
\item[(\textbf{D2})] $\tau_1-i_1\hat d_1-d_{2}+[-\frac{2}{\hat d_1},
    \tau_0-(1-\frac{1}{2(1+i_0)})d_0-i_{1}\hat d_{1}]>0$,

\item[(\textbf{F2})]  $\tau_1-i_1\hat d_1-d_{2}+[-\frac{2}{\hat d_1},
    \tau_0-(1-\frac{1}{2(1+i_0)})d_0-i_{1}\hat d_{1}]<-\frac{d_2}{2i_2}$.
 \end{enumerate}

\vskip.2cm

\noindent $\mathrm{(III)}$.  $0<\cB(DF(V_x^p))<\cB(DF(V_x^d))<1/d_2$ if one
of the following conditions holds:
 \begin{enumerate}
\item[(\textbf{P1})] One of the paired conditions\footnote{In the
    following we will use the term \textbf{(D1)-(D2)}, which is short for
    ``both conditions \textbf{(D1)} and \textbf{(D2)}''. Similarly, we
    use the term \textbf{(D1)-(D2)-(P1)}, which is short for ``all three
    conditions \textbf{(D1)}, \textbf{(D2)} and \textbf{(P1)}''} (that
    is, \textbf{(D1)-(D2)} or \textbf{(F1)-(F2))} holds and
     \beq\label{check} [\tau_1-i_1\hat d_1-d_{2},-\frac{2}{\hat
     d_1},\tau_0-(1-\frac{1}{2(1+i_0)})d_0-i_{1}\hat
    d_{1}]>[\tau_1-i_1\hat d_1-d_{2},-\frac{2}{\hat
    d_1},\tau_0-d_0-i_{1}\hat d_{1}].
      \eeq

\item[\textbf{(P2)}] \textbf{(D1)-(F2)} hold.
 \end{enumerate}
\end{pro}

The statements of Proposition \ref{suffness} is rather technical, but the
proof is quite straightforward. Geometrically, it gives different criterions
when the cone bounded by $V_x^p$ and $V_x^d$ is mapped under $DF$ to the cone
bounded by $V_{Fx}^p$ and $V_{Fx}^d$, see Proposition \ref{returnf} for more
details.

\begin{proof}
Note that $\cB(V^d_x)=1/d_0(x)$ and $\cB(V^p_x)=0$. So the curvatures of
$DF(V_x^d)$ and $DF(V_x^p)$ are given by
\begin{align*}
\cB(DF(V_x^d))&=[d_2,\frac{2i_2}{d_2},\bar\tau_1,-\frac{2}{\hat d_1},\hat\tau_0,
-\frac{2}{\hat d_0},-i_0\hat d_0+d_0]
=[d_2,\frac{2i_2}{d_2},\bar\tau_1,-\frac{2}{\hat d_1},\hat\tau_0-\hat d_0]\\
&=[d_2,\frac{2i_2}{d_2},\bar\tau_1,
-\frac{2}{\hat d_1},\tau_0-d_0-i_{1}\hat d_{1}];\\
\cB(DF(V_x^p))&=[d_2,\frac{2i_2}{d_2},\bar\tau_1,-\frac{2}{\hat d_1},\hat\tau_0,
-\frac{2}{\hat d_0}]=[d_2,\frac{2i_2}{d_2},\bar\tau_1,
-\frac{2}{\hat d_1},\hat\tau_0-\frac{\hat d_0}{2}]\\
&=[d_2,\frac{2i_2}{d_2},\bar\tau_1,
-\frac{2}{\hat d_1},\tau_0-pd_0-i_{1}\hat d_{1}],
\text{ where }p=1-\frac{1}{2(1+i_0)}.
\end{align*}
Here we use the facts that $\hat\tau_0=\tau_0-i_0\hat d_0-i_{1}\hat d_{1}$
and $[\cdots,a,b,0]=[\cdots,a]$ whenever $b\neq 0$. Then it is easy to see
that (I) $0<\cB(DF(V_x^d))<1/d_2$ is equivalent to
 \beq \label{equi1}
[\frac{2i_2}{d_2},\bar\tau_1, -\frac{2}{\hat d_1},\tau_0-d_0-i_{1}\hat
d_{1}]>0.
 \eeq
Moreover, \eqref{equi1}  holds if and only if one of the following conditions
holds:
\begin{itemize}
\item $[\bar\tau_1, -\frac{2}{\hat d_1},\tau_0-d_0-i_{1}\hat d_{1}]>0$,
    which corresponds to \textbf{(D1)};

\item $0>[\bar\tau_1, -\frac{2}{\hat d_1},\tau_0-d_0-i_{1}\hat
    d_{1}]>-\frac{2i_2}{d_2}$, which corresponds to \textbf{(F1)}.
\end{itemize}
We can derive \textbf{(D2)} and \textbf{(F2)} from (II) in the same way.

\vskip.3cm

 Condition (III) $0<\cB(DF(V_x^p))<\cB(DF(V_x^d))<1/d_2$ is
equivalent to
\begin{align}
&[\frac{2i_2}{d_2},\bar\tau_1,
-\frac{2}{\hat d_1},\tau_0-pd_0-i_{1}\hat d_{1}]>[\frac{2i_2}{d_2},\bar\tau_1,
-\frac{2}{\hat d_1},\tau_0-d_0-i_{1}\hat d_{1}]>0.\label{equi2}
\end{align}
Then \eqref{equi2}  holds if and only if one of the following conditions
holds:
\begin{itemize}
\item $0<[\bar\tau_1, -\frac{2}{\hat d_1},\tau_0-pd_0-i_{1}\hat
    d_{1}]<[\bar\tau_1, -\frac{2}{\hat d_1},\tau_0-d_0-i_{1}\hat d_{1}]$,
which is \textbf{(D1)-(D2)-(P1)};

\item $[\bar\tau_1, -\frac{2}{\hat d_1},\tau_0-pd_0-i_{1}\hat
    d_{1}]<[\bar\tau_1, -\frac{2}{\hat d_1},\tau_0-d_0-i_{1}\hat
    d_{1}]<-\frac{d_2}{2i_2}$, which is \textbf{(F1)-(F2)-(P1)};

\item $[\bar\tau_1, -\frac{2}{\hat d_1},\tau_0-d_0-i_{1}\hat d_{1}]>0$
    and $[\bar\tau_1, -\frac{2}{\hat d_1},\tau_0-pd_0-i_{1}\hat
    d_{1}]<-\frac{d_2}{2i_2}$, which is \textbf{(D1)-(F2)}, or equivalently, \textbf{(P2)}.
\end{itemize}
This completes the proof of the proposition.
\end{proof}

\begin{remark}
It is worth pointing out the following observations:
\begin{itemize}
\setlength{\itemsep}{3pt}
\item[(1).] The conditions \textbf{(F1)} and \textbf{(F2)} are empty when
    $i_2=0$.

\item[(2).] The arguments in the proof of Proposition \ref{suffness}
    apply to general billiards with several circular boundary components.
    That is, for any consecutive circular components that a billiard
    orbit passes, say $\Gamma_k$, $k=0,1,2$, let $i_k$ be the number of
    reflections of the orbit on $\Gamma_k$. Then we have the same
    characterizations for $Fx=\cF^{i_0+i_1+i_2+2}x$.

\item[(3).] In the case when all arcs $\Gamma_k$ lie on the same circle,
    the left-hand sides of \textbf{(D1)} and \textbf{(F1)} are zero, and (I) always fails.
    This is quite natural, since the circular billiard is well-known to
    be not hyperbolic, but parabolic.
\end{itemize}
\end{remark}

The following is an application of Proposition \ref{suffness} to asymmetric
lemon tables satisfying  Assumptions \textbf{(A0)}--\textbf{(A3)}.
\begin{pro}\label{returnf}
Let $Q(b,R)$ be an asymmetric lemon billiard satisfying the assumptions
\textbf{(A0)}--\textbf{(A3)}. Then for a.e. $x\in M$,
\begin{equation}\label{cone}
0=\cB(V_{Fx}^p)<\cB(DF(V_x^p))<\cB(DF(V_x^d))<\cB(V_{Fx}^d)=1/d(Fx).
\end{equation}
\end{pro}
In comparing to Proposition \ref{suffness}, the statement of Proposition
\ref{returnf} is very clear, but the proof is quite technical and a little
bit tedious. From the proof, one can see that $\frac{d_0}{1+i_0}<d_1$ is a
very subtle assumption to guarantee $d_0<pd_0+\frac{d_1}{2}$. One can check
that \eqref{cone} may fail if $d_0>pd_0+\frac{d_1}{2}$. This is why we need
$\frac{d_0}{1+i_0}<d_1$ in the assumption (\textbf{A2}).

\begin{proof}
It suffices to show that for a.e. point $x\in M$,  one of the following
combinations holds: \textbf{(D1)-(D2)-(P1)}, or \textbf{(F1)-(F2)-(P1)}, or
\textbf{(D1)-(F2)}.

\vskip.3cm

\noindent{\bf Case 1.} Let  $x\in X_0$. Then Eq. \eqref{assume1} implies that
\[\tau_0-d_0-i_{1}\hat d_{1}<\tau_0-(1-\frac{1}{2(1+i_0)})d_0-i_{1}\hat
d_{1}<0,\quad \tau_1-i_1\hat d_1-d_{2}<-\frac{d_2}{2i_2}.\]
Therefore, \textbf{(F1)-(F2)-(P1)} hold.

\vskip.3cm

\noindent{\bf Case 2.} Let $x\in X_1$. Note that $i_1=0$ and $\hat d_1=d_1$.
We denote $p=1-\frac{1}{2(1+i_0)}$ and $q=1-\frac{1}{2i_2}$, and rewrite the
the corresponding conditions using $i_1=0$:

 \begin{enumerate}
\item[(\textbf{D1}$'$).] $\tau_1-d_{2}+[-\frac{2}{d_1},\tau_0-d_0]>0$, or
    equivalently,
    $\tau_1-d_{2}+\cfrac{1}{-\frac{2}{d_1}+\frac{1}{\tau_0-d_0}}>0$;

\item[(\textbf{F1}$'$).]  $\tau_1-d_{2}+[-\frac{2}{d_1},
    \tau_0-d_0]<-\frac{d_2}{2i_2}$, or equivalently,
    $\tau_1-qd_{2}+\cfrac{1}{-\frac{2}{d_1}+\frac{1}{\tau_0-d_0}}<0$;

\item[(\textbf{D2}$'$).] $\tau_1-d_{2}+[-\frac{2}{d_1}, \tau_0-pd_0]>0$, or
    equivalently,
    $\tau_1-d_{2}+\cfrac{1}{-\frac{2}{d_1}+\frac{1}{\tau_0-pd_0}}>0$;

\item[(\textbf{F2}$'$).]  $\tau_1-d_{2}+[-\frac{2}{d_1},
    \tau_0-pd_0]<-\frac{d_2}{2i_2}$,  or equivalently,
    $\tau_1-qd_{2}+\cfrac{1}{-\frac{2}{d_1}+\frac{1}{\tau_0-d_0}}<0$;

\item[(\textbf{P1}$'$).] $[\tau_1-d_{2}, -\frac{2}{d_1},
    \tau_0-pd_0]<[\tau_1-d_{2},-\frac{2}{d_1}, \tau_0-d_0]$.
\end{enumerate}
Note that $0<\tau_0<d_0+d_1$. There are several subcases when $\tau_0$ varies
in $(0,d_0+d_1)$, see Fig. \ref{subcases}.

\begin{figure}[htb]
\centering
\begin{overpic}[width=140mm]{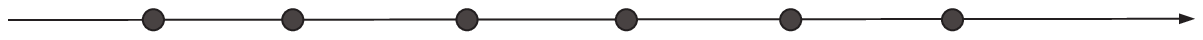}
\put(14,-1){$0$}
\put(24,-1){$pd_0$}
\put(38,-1){$d_0$}
\put(47,-1){$pd_0+\frac{d_1}{2}$}
\put(61,-1){$d_0+\frac{d_1}{2}$}
\put(74,-1){$d_0+d_1$}
\put(94,-1){$\tau_0$}
\put(19,5){(a)}
\put(33,5){(b)}
\put(46,5){(c)}
\put(57,5){(d)}
\put(72,5){(e)}
\end{overpic}
\caption{Subcases of Case 2 according to the values of $\tau_0$.
Note that $d_0<pd_0+\frac{d_1}{2}$,
which follows from the assumption that
$\frac{d_0}{1+i_0}<d_1$.}
\label{subcases}
\end{figure}

\noindent{\bf Subcase (a).} Let $\tau_0< pd_0$. Hence $\tau_1>d_2$ and
$0<\frac{1}{\tau_1-d_2}<\frac{1}{d_0-\tau_0}<\frac{1}{pd_0-\tau_0}$. Then we
claim that \textbf{(D1$'$)-(D2$'$)-(P1$'$)} hold.

\noindent{\it Proof of Claim.} Note that $\tau_0+\tau_1> d_0+d_2$, or equally
$\tau_1-d_2> d_0-\tau_0>0$. So
$\frac{1}{\tau_1-d_{2}}<\frac{1}{d_0-\tau_0}<\frac{2}{d_1}+\frac{1}{d_0-\tau_0}$
which implies (\textbf{D1}$'$). Similarly,
$\frac{1}{\tau_1-d_{2}}<\frac{1}{pd_0-\tau_0}<\frac{2}{d_1}+\frac{1}{pd_0-\tau_0}$
implies (\textbf{D2}$'$). If both terms in (\textbf{P1}$'$) are positive,
then (\textbf{P1$'$}) is equivalent to $[-\frac{2}{d_1},
\tau_0-pd_0]>[-\frac{2}{d_1}, \tau_0-d_0]$. Analogously, if both terms in
(\textbf{P1}$'$) are negative, then (\textbf{P1}$'$) is equivalent to
$[\tau_0-pd_0]<[\tau_0-d_0]$, or equivalently, $\tau_0-pd_0>\tau_0-d_0$,
which holds trivially. This completes the proof of the claim.

\vskip.2cm

\noindent{\bf Subcase (b).} Let  $pd_0\le \tau_0< d_0$. The proof of
(\textbf{D1}$'$) in the previous case is still valid. For (\textbf{D2}$'$),
we note that $\tau_0<d_0<pd_0+\frac{d_1}{2}$ and hence
$-\frac{2}{d_1}+\frac{1}{\tau_0-pd_0}>0$. So (\textbf{D2}$'$) follows. Since
both terms are positive, (\textbf{P1}$'$) follows from $[-\frac{2}{d_1},
\tau_0-pd_0]>0>[-\frac{2}{d_1}, \tau_0-d_0]$.

\vskip.2cm

\noindent{\bf Subcase (c).} Let $d_0\le\tau_0<pd_0+d_1/2$. There are two
subcases:

$\bullet$ $\tau_1> d_2$. The proof is similar to Case (b), and
\textbf{(D1$'$)-(D2$'$)-(P1$'$)} follows.

$\bullet$ $\tau_1\le d_2$. Note that $0<\frac{1}{\tau_0-pd_0} -\frac{2}{d_2}
<\frac{1}{\tau_0-d_0} -\frac{2}{d_2}< \frac{1}{d_{2}-\tau_1}$, from which
\textbf{(D1$'$)-(D2$'$)-(P1$'$)} follows.

\vskip.2cm

\noindent{\bf Subcase (d).} Let $pd_0+d_1/2\le \tau_0<d_0+d_1/2$, which is
equivalent to $\frac{1}{\tau_0-pd_0}-\frac{2}{d_1}<0
<\frac{1}{\tau_0-d_0}-\frac{2}{d_1}$. Using  Assumption \textbf{(A2)} that
$\tau_0+\tau_1>d_0+d_2$, we see that (\textbf{D1}$'$) always holds (as in Subcase
2(c)). Using the condition $\tau_0+\tau_1<pd_0+d_1+qd_2$, we see that
$\tau_1-qd_2<\frac{d_1}{2}$. There are two subcases:

$\bullet$ $\tau_1-qd_2<0$. Then (\textbf{F2}$'$) holds trivially.

$\bullet$ $\tau_1-qd_2>0$. Then $\frac{1}{\tau_1-qd_2}>\frac{2}{d_1}$ implies
(\textbf{F2}$'$).

\vskip.2cm

\noindent{\bf Subcase (e).} Let $\tau_0\ge d_0+d_1/2$. There are also two
subcases:

$\bullet$ $\tau_1-qd_2<0$. Then $\frac{1}{\tau_1-qd_2}<0
<\frac{2}{d_1}-\frac{1}{\tau_0-d_0}< \frac{2}{d_1}-\frac{1}{\tau_0-pd_0}$,
and \textbf{(F1$'$)-(F2$'$)-(P1$'$)} hold.

$\bullet$ $\tau_1-qd_2>0$. Then $\frac{1}{\tau_1-qd_2}>\frac{2}{d_1}$ implies
\textbf{(F1$'$)-(F2$'$)-(P1$'$)}.

\vskip.3cm

\noindent{\bf Case 3.} Let $x\in X_2$. Then $\tau_0<d_1/2$ by assumption. The
proof of Case 2 (a)-(c) also works this case, and
\textbf{(D1$'$)-(D2$'$)-(P1$'$)} hold.
\end{proof}


\subsection{Proof of Theorem \ref{hyper}}\label{provehyper}
Let $\cS_\infty=\bigcup_{n\in\bZ}\cF^{-n}\cS_1$. Note that
$\mu(\cS_\infty)=0$. Let $x\in M\backslash \cS_\infty$. The push-forward
$V_n^p(x)=DF^n(V^p_{F^{-n}x})$ and $V_n^d(x)=DF^n(V^d_{F^{-n}x})$ are well
defined for all $n\ge 1$, and together generate two sequences of tangent
vectors in $T_x\cM$. By Proposition \ref{order} and Lemma \ref{returnf}, we
get the following relation inductively:
\begin{align*}
0=\cB(V_x^p)<\cdots<
\cB(V^p_n(x))&<\cB(V^p_{n+1}(x))\\
&<\cB(V_{n+1}^d(x))<\cB(V^d_n(x))<\cB(V^d_x)=1/d(x).
\end{align*}
Therefore $\ds \cB^u(x):=\lim_{n\to\infty} \cB(V^d_n(x))$ exists, and
$0<\cB^u(x)<1/d(x)$. Let $E^u_x=\langle V^u_x\rangle$ be the corresponding
subspace in $T_x\cM$ for all $x\in M\backslash S_\infty$.

Let $\Phi:\cM\to\cM, (s,\varphi)\mapsto (s,-\varphi)$ be the time reversal
map on $\cM$. Then we have $\cF\circ\Phi=\Phi\circ\cF^{-1}$, and
$F\circ\Phi=\Phi\circ F^{-1}$. In particular, $V^s_x=D\Phi(V^u_{\Phi x})$
satisfies $-1/d(\Phi x)<\cB(V^s_x)=-\cB(V^u_{\Phi x})<0$, which corresponds
to a stable vector for each $x\in M\backslash \cS_{\infty}$. Therefore,
$T_x\cM=E^u_x\oplus E^s_x$ for every point $x\in M\backslash \cS_{\infty}$,
and such a point $x$ is a hyperbolic point of the induced billiard map $F$.

Next we show that the original billiard map $\cF$ is hyperbolic. It suffices
to show that the Lyapunov exponent $\chi^+(\cF,x)>0$ for a.e. $x\in M$, since
$(\cM,\cF,\mu)$ is a (discrete) suspension over $(M,F,\mu_M)$ with respect to
the first return time function, say $\xi_M$, which is given by
$\xi_M(x)=i_0+i_1+i_2+2$ (see \S~\ref{induced}). Note that $\int_M \xi_M
d\mu_M=\frac{1}{\mu(M)}$. So the averaging return time
$\overline{\xi}(x)=\lim_{k\to\infty}\frac{\xi_k(x)}{k}$ exists for a.e. $x\in
M$, where  $\xi_k(x)=\xi_M(x)+\cdots+\xi_M(F^{k-1}x)$ be the $k$-th return
time of $x$ to $M$. Moreover, $1\le \overline{\xi}(x)<\infty$ for a.e. $x\in
M$. Then for a.e. $x\in M$, we have
\[\chi^+(\cF,x)=\lim_{n\to\infty}\frac{1}{n}\log\|D_x\cF^n\|
=\lim_{k\to\infty}\frac{k}{\xi_k(x)}\cdot\frac{1}{k}\log\|D_xF^k\|
=\frac{1}{\overline{\xi}(x)}\cdot\chi^+(F,x)>0.\] This completes the proof.
\qed

\subsection{Proof of Theorem \ref{main}}\label{provemain}
Let $\Gamma_1$ be a major arc of the unit circle with endpoints $A,B$
satisfying $|AB|<1$. For $R>1$, it is easy to see that
$b=(R^2-|AB|^2/4)^{1/2}-(1-|AB|^2/4)^{1/2}>R-1$. Moreover, $b>1$ as long as
$R>2$. In the following, we will assume $R>2$. Therefore, (\textbf{A0})
holds. Then it suffices to show that the assumptions
(\textbf{A1})--(\textbf{A3}) hold for the tables $Q(R)=Q(b,R)$ considered in
Theorem \ref{main}.

\begin{figure}[htb]
\centering
\begin{overpic}[width=50mm]{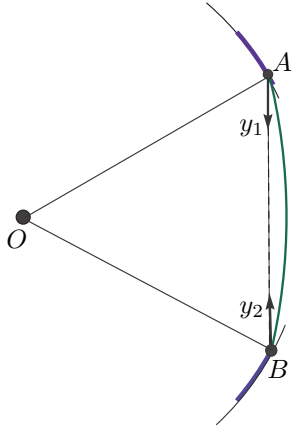}
\put(53,68){$y_1$}
\put(53,29){$y_2$}
\put(3,43){$O$}
\put(60,81){$A$}
\put(59,15){$B$}
\end{overpic}
\caption{First restriction on $R$.
The thickened pieces on $\Gamma_1$ are related to $U$.}
\label{assume-fig0}
\end{figure}
Let $y_1=\overrightarrow{AB}$ and $y_2=\overrightarrow{BA}$ be the two points
in the phase space $\cM$ moving along the chord $AB$, see Fig.
\ref{assume-fig0}. Note that $y_i\in\cS_0$. Since $|AB|<1$, we have $\angle
AOB < \frac{\pi}{3}$, or equally, $n^\ast:=\left\lfloor\frac{2\pi}{\angle
AOB}\right\rfloor\ge 6$, where $\lfloor t\rfloor$ is the largest integer
smaller than or equal to $t$. Then for each $i=1,2$, there is a small
neighborhood $U_i\subset \cM$ of $y_i$, such that for any point $x\in
U_i\cap\cM_1\backslash \cF^{-1}\cM_1$, one has
\begin{itemize}
\item $\cF x$ enters $\cM_R$ and stays on $\cM_R$ for another $i_1$
    iterates, where $i_1=\eta(\cF x)$,

\item $\cF^{i_1+2}x\in M_n$ for some $n\ge n^\ast-1=5$.
\end{itemize}
See Section \ref{cfrac} for the definitions of $\eta(\cdot)$ and $M_n$. Let
$U=(U_1\cup U_2)\cap\cM_1\backslash \cF^{-1}\cM_1$. See Fig.
\ref{assume-fig0}, where the bold pieces on $\Gamma_1$ indicate the bases of
$U$. Note that the directions of vectors in $U$ are close to the vertical
direction and are pointing to some points on $\Gamma_R$.

Before moving on to the verification of the assumptions
(\textbf{A1})--(\textbf{A3}), we need the following lemma. Note that for each
$x\in M$, $\cF^{i_0}x$ is the last reflection of $x\in M$ on $\Gamma_1$
before hitting $\Gamma_R$. That is, $\cF^{i_0}x=\cF^{-1}x_1$. See Fig.
\ref{assume-left} and \ref{assume-right} for two illustrations.
\begin{lem}\label{largeR} For
any $\epsilon>0$, there exists $R_0=R(\epsilon)>2$ such that for any $R\ge
R_0$, the following hold for the induced map $F:M\to M$ of the billiard table
$Q(R)$. That is, for any point $x\in M$ with $d_1\le 4$,
\begin{enumerate}
\item $\cF^{i_0}x\in U$ and $\cF^{i_0+i_1+2}x\in \bigcup_{n\ge 5}M_n$;

\item the reflection points of $\cF^{i_0}x$ and of $\cF^{i_0+i_1+2}x$,
    both on $\Gamma_1$, are $\epsilon$-close to the corners $\{A,B\}$;

\item  the total length of the trajectory from $\cF^{i_0}x$ to
    $\cF^{i_0+i_1+2}x$ is bounded from above by $|AB|+\epsilon$.
\end{enumerate}
\end{lem}
Note that the number 4 in `$d_1\le 4$' is chosen to simplify the presentation
of the proof of Theorem \ref{main}. The above lemma holds for any number
larger than $2|AB|$. In the following we will use $d(x)=\rho\cdot\sin\theta$,
where $\theta=\frac{\pi}{2}-\varphi$ is the angle from the direction of $x$
to the tangent line of $\Gamma_\rho$ at $x$.
\begin{proof}[Proof of Lemma \ref{largeR}]
Let $p_1$ be the reflection point of $x_1$ on $\Gamma_R$, and let $\theta_1$
be the angle between the direction of $x_1$ with the tangent line $L$ of
$\Gamma_R$ at $p_1$. Then $\sin\theta_1\le \frac{4}{R}$, since $d(x_1)=d_1\le
4$. Consider the vertical line that passes through $p_1$. Clearly the length
of the chord on $\Gamma_R$ cut out by this vertical line is less than $|AB|$,
and hence less than 1. Therefore, the angle $\theta_2$ between the vertical
direction with the tangent direction of $\Gamma_R$ at $p_1$ satisfies
$\sin\theta_2< \frac{1}{2R}$.

(1). Note that the angle $\theta$ between $\cF^{i_0}x$ and the vertical
direction is $\theta_1\pm\theta_2$, which satisfies $|\sin
\theta|\le\sin\theta_1 +\sin\theta_2< \frac{5}{R}$. Then $\cF^{i_0}x\in U$ if
$R$ is large enough. Moreover, $\cF^{i_0+i_1+2}x=\cF^{i_1+1}x_1\in
\bigcup_{n\ge 5}M_n$ by our choice of $U$.

(2). Let $P$ be a point on $\Gamma_1$, and $\theta(A)$ be the angle between
the vertical direction with the line $PA$. Similarly we define $\theta(B)$.
Let $\theta_\ast$ be the minimal angle of $\theta(A)$ and $\theta(B)$ among
all possible choices of $P$ on $\Gamma_1$ that is not $\epsilon$-close to
either $A$ or $B$. Then the angle between the vertical direction with the
line connecting $P$ to any point on $\Gamma_R$ is bounded from below by
$\theta_\ast$. Let $x_1\in\cM_R$. The angle $\theta(x_1)$ between the
direction of $x_1$ with the tangent direction of $\Gamma_R$ at the base point
of $x_1$ satisfies $\theta(x_1)\ge \theta_\ast-\arcsin \frac{1}{2R}$. Then
$d_1=R\cdot\sin\theta(x_1)\ge R\sin\theta_\ast-\frac{1}{2}$. So it suffices
to assume $R_0= \frac{5}{\sin\theta_\ast}$.

(3). By enlarging $R_0$ if necessary, the conclusion follows from (2), since
the two reflections are close to the corners and the arc $\Gamma_R$ is almost
flat.
\end{proof}

\begin{figure}[htb]
\centering
\begin{minipage}{0.45\linewidth}
\centering
\begin{overpic}[height=70mm]{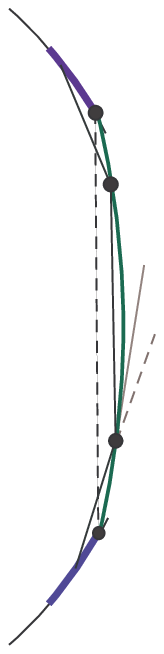}
\put(10,72){$\tau_1$}
\put(10,21){$\tau_0$}
\put(24,56){$L$}
\put(22,33){$p_1$}
\put(18,83){$A$}
\put(18,14){$B$}
\end{overpic}
\caption{The case that there are multiple reflections on $\Gamma_R$.}
\label{assume-left}
\end{minipage}
\begin{minipage}{0.45\linewidth}
\centering
\begin{overpic}[height=70mm]{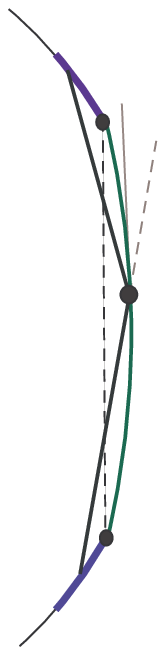}
\put(11,68){$\tau_1$}
\put(11,27){$\tau_0$}
\put(21,79){$L$}
\put(24,54){$p_1$}
\put(18,83){$A$}
\put(18,14){$B$}
\end{overpic}
\caption{The case that there is a single reflection on $\Gamma_R$.}
\label{assume-right}
\end{minipage}
\end{figure}

Now we continue the proof of Theorem \ref{main}.

(1). To verify (\textbf{A1}), we assume $x\in X_0$, which means $i_1\ge1$
(see Fig. \ref{assume-left} for an illustration). Then we have $d_1<|AB|<1$,
which implies $i_0\ge 2$ and $i_2\ge 3$. Then we have
$1-\frac{1}{2(1+i_0)}\ge \frac{5}{6}$, and $1-\frac{1}{2i_2}\ge \frac{5}{6}$.
Therefore, a sufficient condition for (\textbf{A1}) is
\begin{equation}
\tau_0< \frac{5}{6}d_0+\frac{1}{2}d_1,
\quad
\tau_1< \frac{5}{6}d_2+\frac{1}{2}d_1.
\end{equation}
Since $\tau_0,\tau_1<2d_1$, we only need to show that $\frac{3}{4}\tau_0<
\frac{5}{6}d_0$ and $\frac{3}{4}\tau_1< \frac{5}{6}d_2$, or equivalently,
$\tau_0< \frac{10}{9}d_0$ and $\tau_1< \frac{10}{9}d_2$. To this end, we
first set $R_1= 100$: then for each $R\ge R_1$, the angle $\theta$ of the
vertical direction (assume that $AB$ is vertical) and the tangent line of any
point on $\Gamma_R$ satisfies $\sin\theta<\frac{1}{2R}\le \frac{1}{200}$.
Then the angle $\theta_0$ between the vertical direction and the free path
corresponding to $\tau_0$ satisfy $\sin\theta_0\le
\sin\theta+\frac{d_1}{R}\le \frac{1}{100}$ (since $2d_1<1$). This implies
$2d_0\ge (1-0.05)\cdot|AB|$. On the other hand, by making $R_1$ even larger
if necessary, we can assume $\tau_0+\tau_1+2i_1d_1\le (1+0.05)\cdot|AB|$ for
$x\in X_0$ (see Lemma \ref{largeR}). Then we have
\[ \tau_0< \frac{\tau_0+2d_1}{2}<\frac{1.05}{2}\cdot|AB|
\le \frac{1.05}{2}\cdot\frac{2d_0}{0.95}<\frac{10}{9}d_0.
\]
Similarly, we have $\tau_1< \frac{10}{9}d_2$.

(2). To verify (\textbf{A2}), we assume $x\in X_1$, which means $i_1=0$ and
$i_2\ge 1$ (see Fig. \ref{assume-right}). Both relations in (\textbf{A2}) are
trivial if $d_1\ge 4$. In the case $d_1<4$, one must have $i_0\ge 2$, $i_2\ge
3$ (by the assumption $R>R_0$). Note that $|AB|<\tau_0+\tau_1<4d_1$. Pick
$R_2$ large enough, such that for any $R\ge R_2$, for any point $x_1\in
\cM_R$ with $d_1< 4$, the angle $\theta$ of the direction corresponding to
$x_1$ with the vertical direction is bounded by $0.05$ (recall that $AB$ is
vertical). Then we have\footnote{This is a very rough estimate, but is
sufficient for our need.} $2d_0\le \frac{4}{3}\cdot|AB|$. This implies
\[\frac{d_0}{1+i_0}\le \frac{1}{3}d_0\le \frac{2}{9}\cdot|AB|<\frac{8}{9}d_1<d_1.\]
For the second inequality in (\textbf{A2}), by making $R_2$ even larger if
necessary (see Lemma \ref{largeR}), we can assume that
$\tau_0+\tau_1\le1.04\cdot|AB|$, $2d_0,2d_2\ge 0.95\cdot|AB|$. Combining with
the fact that $d_1>|AB|/4$, we have
\[\frac{5}{6}d_0+d_1+\frac{5}{6}d_2>
\frac{5}{6}\cdot \frac{0.95}{2}\cdot|AB|+\frac{1}{4}\cdot|AB|
+\frac{5}{6}\cdot \frac{0.95}{2}\cdot|AB|
>1.04\cdot|AB|.\]
Therefore, $\tau_0+\tau_1<\frac{5}{6}d_0+d_1+\frac{5}{6}d_2$, and
(\textbf{A2}) follows.

(3). The condition (\textbf{A3}) on  $X_2$ holds trivially, since $i_2=0$
implies that $d_1\ge 4>2\tau_0$.

Then we set $R_\ast=\max\{R_i:i=0,1,2\}$. This completes the proof of Theorem
\ref{main}. \qed

\appendix

\section{A Detailed condition}
Now we give an equivalent version of Proposition \ref{suffness}, whose
formulation is longer but easier to check when proving the hyperbolicity of
billiard systems. The statement of the proposition is technical, but the
proof is straightforward.
\begin{pro}
Denote $G_0=\tau_0-d_0-\frac{i_1}{1+i_1}d_1$,
$G_1=\tau_1-\frac{i_1}{1+i_1}d_1-d_2$. Then

\vskip.2cm

\begin{itemize}
\item[(I)] $0<\cB(DF(V_x^d))<1$ if one of the following holds:

\vskip.2cm

\begin{enumerate}
\item[\textbf{(D1a)}] $G_1>0$ and $\frac{1}{G_0}>2\frac{1+i_1}{d_1}$;

\item[\textbf{(D1b)}] $G_1>0$ and
    $\frac{1}{G_0}+\frac{1}{G_1}<2\frac{1+i_1}{d_1}$;

\item[\textbf{(D1c)}] $G_1<0$ and
    $\frac{1}{G_0}+\frac{1}{G_1}<2\frac{1+i_1}{d_1}<\frac{1}{G_0}$;

\item[\textbf{(F1a)}] $G_1+\frac{d_2}{2i_2}<0$ and
    $\frac{1}{G_0}<2\frac{1+i_1}{d_1}$;

\item[\textbf{(F1b)}] $G_1+\frac{d_2}{2i_2}<0$ and
    $\frac{1}{G_0}+\frac{1}{G_1+\frac{d_2}{2i_2}}>2\frac{1+i_1}{d_1}$;

\item[\textbf{(F1c)}] $G_1+\frac{d_2}{2i_2}>0$ and
    $\frac{1}{G_0}+\frac{1}{G_1+\frac{d_2}{2i_2}}
>2\frac{1+i_1}{d_1}>\frac{1}{G_0}$;
\end{enumerate}

\vskip.2cm

\item[(II)] $0<\cB(DF(V_x^p))<1$ if one of the following holds:

\vskip.2cm

\begin{enumerate}
\item[\textbf{(D2a)}] $G_1>0$ and
    $\frac{1}{G_0+\frac{d_0}{2(1+i_0)}}>2\frac{1+i_1}{d_1}$;

\item[\textbf{(D2b)}] $G_1>0$ and
    $\frac{1}{G_0+\frac{d_0}{2(1+i_0)}}+\frac{1}{G_1}<2\frac{1+i_1}{d_1}$;

\item[\textbf{(D2c)}] $G_1<0$ and $\frac{1}{G_0+\frac{d_0}{2(1+i_0)}}
    +\frac{1}{G_1}
    <2\frac{1+i_1}{d_1}<\frac{1}{G_0+\frac{d_0}{2(1+i_0)})}$;
\item[\textbf{(F2a)}] $G_1+\frac{d_2}{2i_2}<0$ and
    $\frac{1}{G_0+\frac{d_0}{2(1+i_0)}}<2\frac{1+i_1}{d_1}$;

\item[\textbf{(F2b)}] $G_1+\frac{d_2}{2i_2}<0$ and
    $\frac{1}{G_0+\frac{d_0}{2(1+i_0)}}
    +\frac{1}{G_1+\frac{d_2}{2i_2}}>2\frac{1+i_1}{d_1}$;

\item[\textbf{(F2c)}] $G_1+\frac{d_2}{2i_2}>0$ and
    $\frac{1}{G_0+\frac{d_0}{2(1+i_0)}}
    +\frac{1}{G_1+\frac{d_2}{2i_2}}
>2\frac{1+i_1}{d_1}>\frac{1}{G_0+\frac{d_0}{2(1+i_0)}}$;
\end{enumerate}

\vskip.2cm

\item[(III)] Eq. \eqref{check} holds if and only if one of the following
    holds:

    \vskip.2cm

\begin{enumerate}
\item[\textbf{(P1a)}]
    $\tau_0<(1-\frac{1}{2(1+i_0)})d_0+(1-\frac{1}{2(1+i_1)})d_1$;

\item[\textbf{(P1b)}]   $\tau_0>d_0+(1-\frac{1}{2(1+i_1)})d_1$.
\end{enumerate}
\end{itemize}
\end{pro}
\begin{proof}
(1). To show \textbf{(D1)}, we let
$Y=\frac{1}{\tau_0-d_0-\frac{i_1}{1+i_1}d_1}-2\frac{1+i_1}{d_1}$. Then
\textbf{(D1)} can be rewritten as $G_1+\frac{1}{Y}>0$, which holds if and
only if one of the following holds:
\begin{enumerate}
\item[(1.1)] $G_1>0$ and $Y>0$;

\item[(1.2)] if $G_1>0$ and $Y<0$, then $\frac{1}{G_1}<-Y$;

\item[(1.3)] if $G_1<0$ and $Y>0$, then $Y<-\frac{1}{G_1}$.
\end{enumerate}
The three conditions \textbf{(D1a)}--\textbf{(D1c)} are evidently equivalent
to three items (1.1)--(1.3), respectively. The verifications of \textbf{(D2)}
and \textbf{(F1)}-\textbf{(F2)} are reduced to the similar calculations and
hence are omitted here.

\vskip.3cm

(2). Let $X=\frac{1}{G_1}-2\frac{1+i_1}{d_1}$ and
\begin{equation}\label{app-B}
Y=\frac{1}{\tau_0-(1-\frac{1}{2(1+i_0)})d_0-\frac{i_1}{1+i_1}d_1}-2\frac{1+i_1}{d_1}.
\end{equation}
Then Eq. \eqref{check} can be rewritten as $\frac{1}{X}<\frac{1}{Y}$, which
is true if and only if one of the following holds:
\begin{enumerate}
\item[(2.1)] if $X<0$ and $Y<0$, then $X>Y$;

\item[(2.2)] if $X>0$ and $Y>0$, then $X>Y$;

\item[(2.3)] $X<0$ and $Y>0$.
\end{enumerate}
Then plugging in the formulas $X=\frac{1}{G_1}-2\frac{1+i_1}{d_1}$ and Eq.
\eqref{app-B} for $Y$, we get that
\begin{enumerate}
\item[(2.1.1)] $\tau_0<(1-\frac{1}{2(1+i_0)})d_0+\frac{i_1}{1+i_1}d_1$;

\item[(2.1.2)] $\tau_0>d_0+(1-\frac{1}{2(1+i_1)})d_1$;

\item[(2.2$'$)]
    $d_0+\frac{i_1}{1+i_1}d_1<\tau_0<(1-\frac{1}{2(1+i_0)})d_0+(1-\frac{1}{2(1+i_1)})d_1$;

\item[(2.3$'$)]
    $(1-\frac{1}{2(1+i_0)})d_0+\frac{i_1}{1+i_1}d_1<\tau_0<\min\{d_0+\frac{i_1}{1+i_1}d_1,
    (1-\frac{1}{2(1+i_0)})d_0+(1-\frac{1}{2(1+i_1)})d_1\}$.
\end{enumerate}
Note that Condition (2.2$'$) is nonempty if and only if
$\frac{d_0}{1+i_0}<\frac{d_1}{1+i_1}$. However, we can always combine
(2.1.1)-(2.2$'$)-(2.3$'$) into one condition:
\[\tau_0<(1-\frac{1}{2(1+i_0)})d_0+(1-\frac{1}{2(1+i_1)})d_1.\]
Therefore we get \textbf{(P1a)-(P1b)}.
\end{proof}

\section*{Acknowledgements}

L.B. is supported in part by  NSF grant DMS-1265883. H.Z. is supported in
part by NSF CAREER grant DMS-1151762. H.Z. would like to thank U. Houston, at
where part of the work was done. P.Z. would like to thank UMass Amherst, at
where part of the work was done. The authors are grateful to the referees for
their careful readings, useful comments and suggestions which helped them to
improve the readability of this paper significantly.


\begin{thebibliography}{s1}


\bibitem{Bu74A} L. A. Bunimovich.
\emph{On billiards closed to dispersing}. {Matem.  Sbornik} {\bf 95} (1974), 49--73.

\bibitem{Bu74B} L. A. Bunimovich.
\emph{On ergodic properties of certain billiards}.
{Funct. Anal. Appl.} \textbf{8} (1974), 254--255.

\bibitem{Bu79} L. A. Bunimovich.
\emph{On the ergodic properties of nowhere dispersing billiards}.
{Commun. Math. Phys.} \textbf{65} (1979), 295--312.



\bibitem{BD} L.A. Bunimovich, G. Del Magno. \emph{Track billiards}.  Comm.
    Mathh.
    Phys.  \textbf{13} (2009) 699--713,

\bibitem{CMZZ} J. Chen, L. Morh, H.-K. Zhang, P. Zhang.
\emph{Ergodicity and coexistence of elliptic islands in a family of convex billiards}.
{Chaos} {\bf 23} (2013), 043137.

\bibitem{CM06} N. Chernov, R. Markarian.
\emph{Chaotic billiards}.
{Math. Surv. Monogr.} {\bf 127}, AMS, Providence, RI, 2006.

\bibitem{CZ05} N. Chernov, H.-K. Zhang.
\emph{Billiards with polynomial mixing rates}.
{Nonlinearity} \textbf{18} (2005), 1527--1553.

\bibitem{Do91} V, Donnay.
\emph{Using integrability to produce chaos: billiards with positive entropy}.
{Commun. Math. Phys.} \textbf{141} (1991), 225--257.

\bibitem{Dou82} R. Douady.
\emph{Applications du th\'eor\`eme des tores invariants}.
Univ. Paris 7, Th\`ese 3-eme cycle, 1982.


\bibitem{HeTo} E. Heller, S. Tomsovic.
\emph{Postmodern quantum mechanics}.
{Physics Today} {\bf 46} (1993), 38--46.

\bibitem{Kh64} A. Ya. Khinchin.
\emph{Continued fractions}.
The University of  Chicago Press, Chicago, 1964.

\bibitem{La} V. F. Lazutkin.
\emph{Existence of a continuum of closed invariant curves for a convex billiard}.
{Math. USSR Izvestija} {\bf 7} (1973), 185--214.



\bibitem{LMR99} V. Lopac, I. Mrkonjic, D. Radic.
\emph{Classical and quantum chaos in the generalized parabolic lemon-shaped billiards}.
Phys. Rev. E 59 (1999), 303--311.

\bibitem{LMR01} V. Lopac, I. Mrkonjic, D. Radic.
 \emph{Chaotic behavior in lemon-shaped billiards with elliptical and hyperbolic boundary arcs}.
Phys. Rev. E 64 (2001) 016214.




\bibitem{M04} R. Markarian.
 \emph{Billiards with polynomial decay of correlations}.
Ergod. Th. Dynam. Sys. {\bf 24} (2004), 177--197.


\bibitem{MHA} H. Makino, T. Harayama, Y. Aizawa.
 \emph{Quantum-classical correspondences of the Berry-Robnik parameter through bifurcations in
lemon billiard systems}.
 Phys. Rev. E {\bf 63} (2001), 056203.


\bibitem{OP05} S. Oliffson Kamphorst, S. Pinto-de-Carvalho.
 \emph{The first  Birkhoff coefficient and the stability of 2-periodic orbits in billiards}.
{Exp. Math.} \textbf{14} (2005), 299--306.

\bibitem{ReRe} S. Ree, L. E. Reichl.
 \emph{Classical and quantum chaos in a circular billiard with a straight cut}.
    {Phys. Rev.} E {\bf 60} (1999),
    1607 (9).


\bibitem{Si70} Ya. Sina\v{\i}.
 \emph{Dynamical systems with elastic reflections. Ergodic properties of diepersing billiards}.
{Russian Math.  Surveys}, {\bf 25} (1970) 137--189.


\bibitem{Woj86} M. Wojtkowski.
 \emph{Principles for the design of billiards with nonvanishing Lyapunov exponents}.
Commun. Math. Phys. {\bf 105} (1986), 391--414.

\end{thebibliography}
\end{document}